\documentclass{amsart}%[oneside, 12pt]{amsart}
\usepackage{tikz}
\usepackage{xcolor}
\usepackage{amssymb,latexsym,amsmath}
\usepackage{amsthm}
\usepackage{graphicx,mathrsfs,comment}
\usepackage{hyperref,url}
\usepackage{pict2e}
\usepackage{enumerate}

\usepackage{bm}

\usepackage{cancel}

\usepackage{amstext}

\numberwithin{equation}{section}

\newcommand{\orcid}[1]{\href{https://orcid.org/#1}{\texttt{ORCID: #1}}}

\usepackage{esint}

\newtheorem{theorem}{Theorem}[section]
\newtheorem{lemma}[theorem]{Lemma}

\newtheorem{proposition}[theorem]{Proposition}
\newtheorem{remark}[theorem]{Remark}
\newtheorem*{remark*}{Remark}

\newtheorem{corollary}[theorem]{Corollary}

\makeatletter
\newcommand{\barredsum}{%
  \DOTSB\mathop{\mathpalette\@barredsum\relax}\slimits@
}
\newcommand{\@barredsum}[2]{%
  \begingroup
  \sbox\z@{$#1\sum$}%
  \setlength{\unitlength}{\dimexpr2pt+\ht\z@+\dp\z@\relax}%
  \@barredsumthickness{#1}%
  \vphantom{\@barredsumbar}%
  \ooalign{$\m@th#1\sum$\cr\hidewidth$#1\@barredsumbar$\hidewidth\cr}%
  \endgroup
}
\newcommand{\@barredsumbar}{%
  \vcenter{\hbox{\begin{picture}(0,1)\roundcap\Line(0,0)(0,1)\end{picture}}}%
}
\newcommand{\@barredsumthickness}[1]{% see https://tex.stackexchange.com/a/477200/
  \linethickness{%
    1.25\fontdimen8
      \ifx#1\displaystyle\textfont\else
      \ifx#1\textstyle\textfont\else
      \ifx#1\scriptstyle\scriptfont\else
      \scriptscriptfont\fi\fi\fi 3
  }%
}
\makeatother

\newcommand{\cB}{{\mathcal B}}

\newcommand{\R}{\mathbb{R}}

\newcommand{\dist}{{\rm dist}}

\newcommand{\supp}{{\rm supp}}

\usepackage{mathtools,booktabs,microtype}
\allowdisplaybreaks[2]
\newcommand{\Rap}{\operatorname{RapDec}}
\newcommand{\norm}[2]{\left\lVert #1\right\rVert_{#2}}

\newcommand{\Sphere}{\mathbb S}
\newcommand{\dd}{\,d}
\hypersetup{pdftitle={Cone Restriction Estimates under Half-Scale Dimensional Conditions},
 pdfauthor={Xiangyu Wang},
 colorlinks=false,
 pdfborder={0 0 1},
 linkbordercolor={1 0 0},
 citebordercolor={0 0.7 0},
 urlbordercolor={0 0.7 0}}

\begin{document}
\title[Cone Restriction Estimates under Half-Scale Dimensional Conditions]
{Cone Restriction Estimates under Half-Scale Dimensional Conditions}
\author{Xiangyu Wang}
\address{Xiangyu Wang \\ Department of Mathematics, University of Illinois Urbana-Champaign\\ \orcid{0009-0003-5983-6961}}
\email{xw70@illinois.edu}
\begin{abstract}
We prove cone restriction estimates under half-scale dimensional conditions. Building on the framework of Du--Li--Wang--Zhang, we incorporate the broad restriction estimate of Ou--Wang and a localized multilinear estimate, which is obtained from cyclic bilinear pairing and multilinear Kakeya bounds for the local $L^2$ energies of wave packets. The zero-dimensional case yields a large-value estimate on $R^{1/2}$-cubes. For certain ambient and weight dimensions, our estimates are sharp, up to an $R^\varepsilon$ loss, in suitable ranges of $p$.
\end{abstract}

\maketitle

\section{Introduction}\label{sec:intro}
For $n\ge2$, let $E$ denote the Fourier extension operator associated with the truncated cone in $\R^{n+1}$, defined by
\[
 Ef(x,t):=\int_{1\le|\xi|\le2}
       e^{i(x\cdot\xi+t|\xi|)}f(\xi)\dd\xi,
 \qquad (x,t)\in\R^n\times\R.
\]
We write $B_R$ for an arbitrary ball of radius $R$ in $\R^{n+1}$ and, for a weight function $H$, we write
\[
 \norm{Ef}{L^p(B_R,H)}
 =\left(\int_{B_R}|Ef(z)|^pH(z)\dd z\right)^{1/p}.
\]
Throughout the paper, $m$ denotes the \emph{weight dimension}; it need not be an integer. For $0\le m\le n-1$, our hypothesis on the weight is
\begin{equation}\label{eq:weight}
 \begin{gathered}
 0\le H\le1,\qquad
 \int_{B(z,r)}H(w)\dd w\le C_Hr^m,\\
 z\in\R^{n+1},\qquad 1\le r\le R^{1/2}.
 \end{gathered}
\end{equation}
The ball condition is imposed only at scales up to $R^{1/2}$; we call \eqref{eq:weight} a \emph{half-scale dimensional condition}. In particular, the total mass of $H$ in $B_R$ need not be $O(C_HR^m)$: covering $B_R$ by balls of radius $R^{1/2}$ gives only
\[
 \int_{B_R}H\lesssim C_HR^{(n+1+m)/2}.
\]

Weighted estimates of this kind are closely related to the sparsity of the large values of $Ef$. Indeed, let $\lambda>0$, and let $X\subset B_R$ be a union of unit balls on which $|Ef|\ge\lambda$. If $H=\mathbf1_X$ satisfies \eqref{eq:weight}, then an estimate of the form $\norm{Ef}{L^p(B_R,H)}\lesssim C_H^{1/p}R^\beta\norm f2$ gives
\[
 |X|\lesssim C_H\left(\frac{R^\beta\norm f2}{\lambda}\right)^p.
\]
Thus the estimate limits the size of any part of the large-value set $\{|Ef|\ge\lambda\}$ that satisfies the ball condition in \eqref{eq:weight}: the large values of $Ef$ must be sparse. This is most transparent when $m=0$, where the weight sees only boundedly many unit balls in each ball of radius $R^{1/2}$ and the estimate counts the $R^{1/2}$-cubes on which $Ef$ is large.

Du--Li--Wang--Zhang \cite{DLWZ} proved weighted estimates for the paraboloid under the same half-scale condition, by combining multilinear refined Strichartz estimates with a broad--narrow analysis and a discrete induction that keeps track of the number of cubes and of their density. We build on this framework and incorporate two further ingredients: the broad restriction estimate of Ou--Wang \cite{OW}, and a localized multilinear estimate obtained from cyclic bilinear pairing and multilinear Kakeya bounds for the local $L^2$ energies of wave packets. The zero-dimensional case $m=0$ yields a large-value estimate on $R^{1/2}$-cubes (Corollary~\ref{cor:large-values}).

Our necessary conditions come from three families of examples: lightplanks; lattices in null coordinates, following \cite[Section 3]{DLWZ}; and hybrid lattice--Knapp examples built from integer points on lower-dimensional cones, as in \cite[Proposition 5.5]{HarrisRS} and \cite[Section 5]{DHL}. The weights in all of these examples satisfy \eqref{eq:weight}. For certain ambient and weight dimensions, the examples show that our estimates are sharp, up to an $R^\varepsilon$ loss, in suitable ranges of $p$; we do not determine the optimal exponent for every choice of the parameters.

\subsection{The main estimates}
Write $a_+=\max\{a,0\}$. For $3\le k\le n+1$, let
\begin{equation}\label{eq:restriction-exponents}
 r_k=\frac{2(k-1)}{k-3},\qquad
 q=\frac{2(n+1)}{n-1},\qquad
 \bar p_k=\frac{2(n+k+1)}{n+k-1}.
\end{equation}
When $k=3$, we interpret $r_k$ as $\infty$. We have $r_k\ge q\ge2$; the admissible range of $p$ is specified in each statement.
Define
\[
\begin{aligned}
 A_{n,m,k}(p)
   &=\frac1{2k}+\frac{m}{2(n+1)}
        -\frac{n+1+m}{2}\left(\frac12-\frac1p\right),\\
 N_{n,m,k}(p)
   &=\frac12-\frac{n-m}{2(k-1)}
        -\frac{m+1}{2}\left(\frac12-\frac1p\right),\\
 C_{n,m}(p)
   &=\frac{m}{2(n+1)}-\frac m2\left(\frac12-\frac1p\right).
\end{aligned}
\]
The exponents $A_{n,m,k}$ and $C_{n,m}$ arise from the two extreme cases of the counting argument based on the refined Strichartz estimate (Section~\ref{subsec:broad-strichartz}), and $N_{n,m,k}$ is the lower bound on the exponent that is required by the narrow step of the induction (Section~\ref{subsec:narrow-step}).

For $n\ge3$, the localized multilinear estimate involves the exponents
\[
 v_k=\min\left\{\frac{2(n+3)}{n+1},\frac{2k}{k-1}\right\}
\]
and
\begin{equation}\label{eq:local-exponents}
 s_0=\max\left\{2,\frac{8m}{n-1+2m}\right\},\qquad
 \frac1{s_k}=\max\left\{\frac1{v_k},\
 \frac1{r_k}+\frac{1-v_k/r_k}{s_0}\right\}.
\end{equation}
They satisfy $2\le s_k\le v_k\le3$ and $v_k<r_k$, and they depend only on $n$, $m$, and $k$.

\begin{theorem}\label{thm:full}\label{thm:main}\label{thm:bilinear-refinement}
Let $n\ge2$ and $0\le m\le n-1$, and suppose that $H$ satisfies \eqref{eq:weight}. Then, for every $\varepsilon>0$, the estimate
\begin{equation}\label{eq:full}
 \norm{Ef}{L^p(B_R,H)}
 \lesssim_{n,m,p,k,\varepsilon}
 C_H^{1/p}R^{\beta+\varepsilon}\norm f2
\end{equation}
holds in each of the following cases:

\textup{(i)} For $3\le k\le n+1$ and $2\le p\le q$,
\[
 \beta=\max\{A_{n,m,k}(p),N_{n,m,k}(p),C_{n,m}(p)\}.
\]

\textup{(ii)} For $3\le k\le n+1$ and $2\le p\le r_k$, with $p<\infty$,
\[
 \beta=\max\left\{N_{n,m,k}(p),\
 \frac{n+1+\min\{m,k\}}2
       \left(\frac1p-\frac1{\bar p_k}\right)_+,\ 0\right\}.
\]

\textup{(iii)} For $n\ge3$, $3\le k\le n+1$, and $2\le p\le\min\{3,r_k\}$,
\[
 \beta=\max\left\{N_{n,m,k}(p),\
 \left(\frac{k}{2p}-\frac{k-1}{4}\right)_+
 +\frac m2\left(\frac1p-\frac1{s_k}\right)_+,\ 0\right\}.
\]
The implicit constant is independent of $R\ge1$, of the center of $B_R$, of $f$, and of $H$. Moreover, for $p\ge q$ the estimate holds with $\beta=0$ and without the $R^\varepsilon$ loss.
\end{theorem}

For $2\le p<\infty$, let $b_{n,m}(p)$ denote the smallest of the exponents $\beta$ that Theorem~\ref{thm:main} provides at $p$, where the exponent $0$ is included when $p\ge q$, and put $b_{n,m}(\infty)=0$. Interpolating between these bounds leads to the exponent
\begin{equation}\label{eq:convex}
 \begin{split}
 B_{n,m}(p)=\inf\Bigl\{&\sum_{j=1}^J\lambda_jb_{n,m}(p_j):\quad
 \lambda_j\ge0,\quad\sum_j\lambda_j=1,\\
 &\sum_j\lambda_j/p_j=1/p,\quad
 2\le p_j\le\infty,\quad J<\infty\Bigr\}.
 \end{split}
\end{equation}
Here $1/\infty=0$. Thus $B_{n,m}$ is the lower convex envelope of $b_{n,m}$, regarded as a function of $1/p$; it need not be the optimal exponent in \eqref{eq:full}. By interpolation,
\begin{equation}\label{eq:combined-upper}
 \norm{Ef}{L^p(B_R,H)}
 \lesssim_{n,m,p,\varepsilon}C_H^{1/p}
 R^{B_{n,m}(p)+\varepsilon}\norm f2,
 \qquad 2\le p<\infty.
\end{equation}

\subsection{Consequences}
Taking $m=0$ and $k=n+1$ in Theorem~\ref{thm:main}\textup{(i)}, we obtain the following large-value estimate.
\begin{corollary}[Large values on $R^{1/2}$-cubes]\label{cor:large-values}
Let $n\ge2$, $R\ge1$, and
\[
 p_0=\frac{2(n+1)^2}{(n+1)^2-2}.
\]
Let $\mathcal Q$ be the family of lattice cubes of side $R^{1/2}$ that meet $B_R$. Then, for every $\varepsilon>0$,
\[
 \sum_{Q\in\mathcal Q}
 \left(\sup_{z\in Q\cap B_R}|Ef(z)|\right)^{p_0}
 \lesssim_{n,\varepsilon}R^\varepsilon\norm f2^{p_0}.
\]
Consequently, for every $\lambda>0$,
\[
 \#\left\{Q\in\mathcal Q:
       \sup_{z\in Q\cap B_R}|Ef(z)|\ge\lambda\right\}
 \lesssim_{n,\varepsilon}R^\varepsilon
       \left(\frac{\norm f2}{\lambda}\right)^{p_0}.
\]
\end{corollary}

In particular, if $\norm f2=1$ and $c>0$ is fixed, then Corollary~\ref{cor:large-values} gives
\[
 \#\left\{Q\in\mathcal Q:
       \sup_{z\in Q\cap B_R}|Ef(z)|\ge cR^{-(n-1)/4}\right\}
 \lesssim_{n,\varepsilon,c}R^{(n-1)p_0/4+\varepsilon}.
\]
For a sufficiently small fixed $c>0$, the normalized Knapp example of Section~\ref{subsec:lightplank} reaches this threshold on $\sim R^{1/2}$ of these cubes. In $\R^3$, that is, for $n=2$, the upper bound is $R^{9/14+\varepsilon}$, whereas this example gives the lower bound $R^{1/2}$; the two exponents differ by $9/14-1/2=1/7$.

An analogous estimate holds for the paraboloid. Let $E_{\mathrm P}$ denote the extension operator associated with the paraboloid in $\R^{n+1}$. Applying the argument of Du--Li--Wang--Zhang \cite[Section~4]{DLWZ} to zero-dimensional half-scale weights gives
\[
  \norm{E_{\mathrm P}f}{L^{p_{\mathrm P}}(B_R,H)}
  \lesssim_{n,\varepsilon}
  C_H^{1/p_{\mathrm P}}R^\varepsilon\norm f2,
  \qquad p_{\mathrm P}=2+\frac{4}{n(n+3)},
\]
whenever $H$ satisfies \eqref{eq:weight} with $m=0$.
This yields the analogue of Corollary~\ref{cor:large-values} with
$p_{\mathrm P}$ in place of $p_0$. In particular, if $\norm f2=1$
and $c>0$ is fixed, then
\[
  \#\left\{Q\in\mathcal Q:
    \sup_{z\in Q\cap B_R}|E_{\mathrm P}f(z)|\ge cR^{-n/4}\right\}
  \lesssim_{n,\varepsilon,c}
  R^{\frac n2+\frac1{n+3}+\varepsilon}.
\]
At this amplitude, the version of Bourgain's entropy conjecture in general dimensions suggested by Zhang \cite[Section~6.1.2]{ZhangICM} predicts the
upper bound $R^{n/2+\varepsilon}$. Thus our exponent exceeds the conjectured one by $1/(n+3)$, which tends to zero as $n\to\infty$.

We also obtain the following sharp $L^2\to L^p$ estimates; the examples that show their sharpness are given in Section~\ref{sec:examples}. In the remaining corollaries of this subsection, $H$ is a weight satisfying \eqref{eq:weight}.

\begin{corollary}[A sharp range in $\R^3$]\label{cor:sharp-R3}
Let $n=2$ and $1/2\le m\le1$, and set
\[
 p(m)=
 \begin{cases}
  12/(2m+3),&1/2\le m<2/3,\\[2pt]
  12/(m+4),&2/3\le m\le1.
 \end{cases}
\]
Then, for $p(m)\le p<\infty$,
\[
 \norm{Ef}{L^p(B_R,H)}
 \lesssim_{m,p,\varepsilon}C_H^{1/p}
 R^{\max\{0,(m+1)/(2p)-1/4\}+\varepsilon}\norm f2.
\]
This power of $R$ is sharp up to the $R^\varepsilon$ loss. In particular, when $m=1$ the exponent is $1/p-1/4$ for $12/5\le p\le4$ and zero for $p\ge4$.
\end{corollary}

\begin{corollary}[A sharp range in $\R^4$]\label{cor:sharp-R4}
Let $n=3$ and $m=2$. Then
\[
 \norm{Ef}{L^p(B_R,H)}\lesssim_{p,\varepsilon}C_H^{1/p}
 R^{\max\{0,3/(2p)-1/2\}+\varepsilon}\norm f2,
 \qquad 21/8\le p<\infty.
\]
This power of $R$ is sharp up to the $R^\varepsilon$ loss. Interpolation also gives
\begin{equation}\label{eq:R4-refined-originalp}
 \norm{Ef}{L^{12/5}(B_R,H)}
 \lesssim_\varepsilon C_H^{5/12}R^{29/184+\varepsilon}\norm f2.
\end{equation}
We do not claim that the exponent $29/184$ is sharp.
\end{corollary}

\begin{corollary}[Zero-power estimates]\label{cor:zero-power}
Let $n\ge2$ and $0<m\le n-1$, and define
\[
 p_\circ=\min_{3\le k\le n+1}
 \max\left\{\bar p_k,\
 \frac{2(m+1)}{m-1+2(n-m)/(k-1)}\right\}.
\]
Then, for $p_\circ\le p<\infty$,
\begin{equation}\label{eq:zero-power}
 \norm{Ef}{L^p(B_R,H)}
 \lesssim_{n,m,p,\varepsilon}C_H^{1/p}R^\varepsilon\norm f2.
\end{equation}
For $n\ge4$, the estimate \eqref{eq:zero-power} holds in particular for every $p\ge2(n+4)/(n+2)$. When $n=4$ and $m=3$, one has $p_\circ=8/3$, and \eqref{eq:zero-power} fails for every $p<8/3$.
\end{corollary}

Finally, for small $p$ we obtain the following estimate.

\begin{corollary}[A small-$p$ estimate]\label{cor:small-p-general}
Let $n\ge2$ and $0\le m\le n-1$, and set
\begin{equation}\label{eq:small-p-general-endpoint}
 p_{n,m}=\left(\frac12-
 \min\left\{\frac{n-m}{n^2(n+1)},\frac1{(n+1)^2}\right\}\right)^{-1}.
\end{equation}
Then, for $2\le p\le p_{n,m}$,
\[
 \norm{Ef}{L^p(B_R,H)}
 \lesssim_{n,m,p,\varepsilon}C_H^{1/p}
 R^{\frac{n+1+m}{2}(\frac1p-\frac12)
       +\frac{m+1}{2(n+1)}+\varepsilon}\norm f2.
\]
If $1\le m\le n-1$, then
\[
 p_{n,m}=\left(\frac12-\frac{n-m}{n^2(n+1)}\right)^{-1}.
\]
\end{corollary}

\subsection{Organization}
Section~\ref{sec:preliminaries} collects the analytic estimates that we need. Section~\ref{sec:discrete} reduces Theorem~\ref{thm:main} to a discrete estimate, Proposition~\ref{prop:discrete}, which is proved by induction on the scale: Section~\ref{sec:narrow} treats the narrow case, which is common to all three parts of the theorem, and Section~\ref{sec:broad} proves the three broad bounds. Section~\ref{sec:optimization} optimizes the exponents, and Section~\ref{sec:examples} establishes the necessary conditions and proves the sharpness assertions.

\subsection{Conventions}\label{subsec:conventions}
We write $X\lesssim Y$ if $X\le CY$ for a constant $C$ that is independent of the scale and of the functions under consideration. The dependence of $C$ on fixed parameters is suppressed unless it needs to be tracked. Unspecified dimensional constants are denoted by $O(1)$, whereas $C_H$ and the induction constant $C_\varepsilon$ are always written explicitly.

For nonnegative $X$ and $Y$, we write $X\lessapprox Y$ if $X\lesssim(\log(2+R))^{O(1)}Y$, we write $X\gtrapprox Y$ if $Y\lessapprox X$, and we write $X\approx Y$ if both hold. Thus these symbols suppress only logarithmic factors, whereas $\sim$ denotes comparability up to fixed constants. The implicit constants and the exponents of the logarithms are independent of the current radius and of the induction constant. Losses in the broad case, and other losses that are incurred only once, are recorded as $K^{O(1)}$; the losses in the narrow case, which are incurred at every step of the induction, are recorded as $K^{O(\delta_1)+O(\eta)}$.

We fix $0<\delta_3<1/10$ and choose the remaining small parameters so that
\[
 0<\delta_1\ll\delta_2\ll\eta\ll\min\{\varepsilon,\delta_3\}.
\]
Each parameter is chosen sufficiently small in terms of the parameters to its right and of the fixed dimensions. Throughout, $K=R^{\delta_1}$. The number $A$ of exceptional planes in the broad norms of Section~\ref{subsec:broad-norms} is fixed, in terms of $n$, $k$, $\varepsilon$, and the exponents, before $\delta_1$ is chosen; it is independent of $R$, and implicit constants may depend on it. For the Ou--Wang estimate, $\delta_1$ is chosen in accordance with the parameter hierarchy described in Section~\ref{subsec:broad-norms}. The other parameters enter as follows: the spatial localization in Section~\ref{subsec:charts} uses a padding factor $K^{\delta_1}$, the wave packet tubes in Section~\ref{subsec:local-packets} have radius $R^{1/2+\delta_2}$, and the radial frequency margins have width comparable to $R^{-\delta_3}$.

Lattice cubes have pairwise disjoint interiors; we allow fixed dilations and rotations of the lattice, as well as finitely many of its translates. In pigeonholing arguments, each quantity is normalized by the relevant $L^2$ norm and the classes with negligible contribution are discarded, so that only $\lessapprox1$ dyadic values remain. Finally, $\Rap(R)$ denotes a quantity that is bounded by $C_LR^{-L}$, where the decay order $L$ may be prescribed arbitrarily large and the constant $C_L$ may depend on $L$ and on $\delta_1$.

\subsection{Acknowledgments}
The author would like to thank Prof.~Xiaochun Li for suggesting this research topic. Claude was used for language refinement, LaTeX typesetting, and manuscript organization. ChatGPT performed most of the calculations in the manuscript, including the optimization of exponents and numerical results. Through interactive discussions, the author and ChatGPT developed the argument for part \textup{(iii)} of Theorem~\ref{thm:main}. ChatGPT also provided technical assistance with the construction and analysis of the Knapp-type and hybrid lattice--Knapp counterexamples in Section~\ref{sec:examples}.

\section{Preliminaries}\label{sec:preliminaries}\label{section-preliminary}
\subsection{Local averaging and finite translations}
The space-time Fourier support of $Ef$ is contained in a fixed compact set, and
\begin{equation}\label{eq:Linfty}
 \norm{Ef}{L^\infty(\R^{n+1})}\lesssim\norm f2.
\end{equation}
Choose $\psi\in\mathcal S(\R^{n+1})$ with $\widehat\psi=1$ on a neighborhood of this set, so that $Ef=(Ef)*\psi$. By H\"older's inequality and the rapid decay of $\psi$, for every cube $B$ of side $K^2$ and every $1\le s\le\infty$ we have
\begin{equation}\label{eq:localaverage}
 \sup_B|Ef|\lesssim_s\norm{Ef}{L^s(10B)}
       +\Rap(R)\norm f2.
\end{equation}
The error term comes from the tail of $\psi$, which decays to arbitrary order in $K$ and hence in $R$; the order of decay that is needed depends on $\delta_1$.

The following lemma combines the two conversions between local norms in \cite[(4.5)--(4.9)]{DLWZ}.
\begin{lemma}[Finite-translation product estimate]\label{lem:translations}
Let $3\le k\le n+1$, let $r_k$ be as in \eqref{eq:restriction-exponents}, and let $0<p\le r_k$ with $p<\infty$. There is a fixed lattice spacing $a_0>0$ with the following property. For every cube $B$ of side $K^2$ and all functions $f_1,\ldots,f_k$, there exist translations
\[
 v_j\in a_0\mathbb Z^{n+1},\qquad |v_j|\lesssim K^2,
 \qquad j=1,\ldots,k,
\]
such that
\begin{equation}\label{eq:translation-product}
 \prod_{j=1}^k\norm{Ef_j}{L^{r_k}(B)}^{1/k}
 \lesssim K^{O(1)}
 \left\|\prod_{j=1}^k|Ef_j(\,\cdot+v_j)|^{1/k}\right\|_{L^p(B)}
 +\Rap(R)\prod_{j=1}^k\norm{f_j}2^{1/k}.
\end{equation}
The number of possible tuples $(v_1,\ldots,v_k)$ is $K^{O(1)}$.
\end{lemma}
\begin{proof}
Choose $a_0$ so small that the Fourier support of every $Ef_j$ lies strictly inside the cube $(-\pi/a_0,\pi/a_0)^{n+1}$. Fix a smooth cutoff function that is supported in this cube and equals one near these supports, and let $\psi_0$ be its inverse Fourier transform. The Poisson summation formula gives the oversampled reconstruction formula
\[
 Ef_j(z)=a_0^{n+1}\sum_{v\in a_0\mathbb Z^{n+1}}
                 Ef_j(x+v)\psi_0(z-x-v).
\]
The series converges absolutely, because $Ef_j$ is bounded and $\psi_0$ is a Schwartz function. Consequently, for every $s>0$, uniformly in $x\in B$,
\[
 \sup_{z\in B}|Ef_j(z)|^s
 \lesssim_s\sum_{\substack{v\in a_0\mathbb Z^{n+1}\\|v|\lesssim K^2}}
 |Ef_j(x+v)|^s+\Rap(R)\norm{f_j}2^s.
\]
For $s\le1$ this follows from subadditivity, and for $s>1$ from a weighted H\"older inequality. The omitted tail of the lattice sum decays to arbitrary order in $K$, and hence in $R$, once the order of decay of $\psi_0$ is chosen suitably.

Apply this with $s=p/k$, take the product over $j$, and integrate over $x\in B$. Since $\norm{Ef_j}{L^{r_k}(B)}\le|B|^{1/r_k}\sup_B|Ef_j|$, the $p$-th power of the left-hand side of \eqref{eq:translation-product} is bounded by $K^{O(1)}$ times
\[
 \sum_{v_1,\ldots,v_k}\int_B
             \prod_{j=1}^k|Ef_j(x+v_j)|^{p/k}\dd x
\]
plus a negligible error. Choosing the largest term in the sum proves the lemma, including the case $r_k=\infty$.
\end{proof}

Multiplying $f_j(\xi)$ by $e^{i(v_j'\cdot\xi+v_{j,n+1}|\xi|)}$, where $v_j=(v_j',v_{j,n+1})$, replaces $Ef_j$ by $Ef_j(\cdot+v_j)$ and preserves both the support and the $L^2$ norm of $f_j$.

\subsection{The multilinear refined Strichartz estimate}
For $\xi\ne0$, let
\[
 G(\xi)=2^{-1/2}(-\xi/|\xi|,1)
\]
be the unit normal to the cone at the point $(\xi,|\xi|)$. We say that functions $f_1,\ldots,f_k$ are \emph{$\nu$-transverse} if
\[
 |G(\xi_1)\wedge\cdots\wedge G(\xi_k)|\ge\nu
 \quad\text{for all }\xi_j\in\supp f_j.
\]
The following is the multilinear refined Strichartz estimate for the cone from \cite[Theorem 4.2]{HarrisRS}, stated with a quantitative dependence on the transversality.

\begin{theorem}\label{thm:multilinear}
Let $q=2(n+1)/(n-1)$ and $2\le k\le n+1$. Suppose that $f_1,\ldots,f_k$ are $\nu$-transverse and that $\mathcal Q$ is a family of $N$ lattice $R^{1/2}$-cubes in $B_R$. Assume that, for each $j$, the quantity $\norm{Ef_j}{L^q(Q)}$ is dyadically constant as $Q$ ranges over $\mathcal Q$. Then, for every $\eta>0$,
\[
 \left\|\prod_{j=1}^k|Ef_j|^{1/k}\right\|_{L^q(\bigcup\mathcal Q)}
 \lesssim_\eta \nu^{-O(1)}R^\eta N^{-\frac{k-1}{k(n+1)}}
                  \prod_{j=1}^k\norm{f_j}2^{1/k}.
\]
The same estimate holds in fixed cone charts and with the cubes replaced by fixed enlargements.
\end{theorem}
The polynomial dependence on $\nu$ is obtained by keeping track of the constant in the multilinear Kakeya estimate in the proof of \cite[Theorem 4.2]{HarrisRS}. In our application $\nu\ge K^{-O(1)}$, so this dependence costs a factor $K^{O(1)}$, which occurs only in the broad case.

\subsection{Broad norms}\label{subsec:broad-norms}
Decompose $f=\sum_\tau f_\tau$ using a smooth partition of unity adapted to boundedly overlapping $K^{-1}$-caps $\tau$, and set $G(\tau)=\{G(\xi):\xi\in\tau\}$. For a linear subspace $V$, let $\angle(G(\tau),V)$ denote the smallest angle between a vector in $G(\tau)$ and a nonzero vector in $V$. Given $A\in\mathbb N$, $1\le s<\infty$, and a cube $B$ of side $K^2$, define
\[
 \mu_{Ef,s,A}(B)
 =\min_{V_1,\ldots,V_A}
   \max_{\substack{\tau:\ \angle(G(\tau),V_a)>K^{-2}\\1\le a\le A}}
       \int_B|Ef_\tau|^s,
\]
where the minimum is taken over all $A$-tuples of $(k-1)$-dimensional linear subspaces of $\R^{n+1}$, and the maximum over an empty family is defined to be zero. We refer to $V_1,\ldots,V_A$ as \emph{exceptional planes}. For a union $U$ of such cubes, put
\[
 \norm{Ef}{BL^s_{k,A}(U)}
 =\left(\sum_{B\subset U}\mu_{Ef,s,A}(B)\right)^{1/s}.
\]
When $U$ is a ball, the sum is taken over the cubes that meet it; the resulting enlargement of the ball by a fixed factor is harmless. We also allow fixed dilates of the cubes and families of cubes with bounded overlap.

\begin{theorem}[Ou--Wang; {\cite[Theorem 3]{OW}}]\label{thm:OW-original}
Let $n\ge2$ and $2\le k\le n+1$. For every $\eta>0$, one can fix an integer $A_0=A_0(n,k,\eta)$ and then choose $\delta_1>0$ sufficiently small so that, with $K=R^{\delta_1}$ and $R$ sufficiently large,
\[
 \norm{Ef}{BL^s_{k,A_0}(B_R)}
 \lesssim_{n,k,s,\eta} R^\eta\norm f2,
 \qquad s\ge\frac{2(n+k+1)}{n+k-1},\quad s<\infty.
\]
\end{theorem}
We use the parameter hierarchy $1\ll A_0\lesssim K^\eta\lesssim R^{\eta^2}$ stated before \cite[Theorem 3]{OW}: first fix $A_0$, then take $\delta_1$ sufficiently small, and finally take the base radius sufficiently large. The ambient dimension in \cite{OW} corresponds to $n+1$ in our notation, and the angle threshold there is $K^{-2}$, as in our definition.

For $3\le k\le n+1$, Plancherel's theorem gives
\[
 \norm{Ef}{BL^2_{k,1}(B_R)}^2
 \le\sum_\tau\norm{Ef_\tau}{L^2(B_{O(R)})}^2
 \lesssim R\norm f2^2.
\]
We interpolate between this bound and Theorem~\ref{thm:OW-original} at $s=\bar p_k$, using H\"older's inequality for broad norms \cite[Lemma 3.2]{OW}; the exceptional planes of the two estimates are combined, so that $A$ remains independent of $R$. For $s\ge\bar p_k$, we interpolate instead with the pointwise bound \eqref{eq:Linfty}. Thus, for a sufficiently large fixed $A$ and a sufficiently small $\delta_1$,
\begin{equation}\label{eq:OW-interpolated}
 \norm{Ef}{BL^s_{k,A}(B_R)}
 \lesssim_\eta
 R^{\frac{n+k+1}{2}(\frac1s-\frac1{\bar p_k})_++\eta}
 \norm f2,
 \qquad 2\le s<\infty.
\end{equation}
A fixed dilate of a cube is covered by finitely many $K^2$-cubes. Taking the union of their exceptional planes, and increasing $A$ by a fixed factor, we see that \eqref{eq:OW-interpolated} extends to the fixed dilates and to the boundedly overlapping families of cubes used below.

\subsection{Narrow decoupling}
A $K^{-1}$-cap has angular width comparable to $K^{-1}$ and radial length comparable to one. The narrowness hypothesis in the next theorem requires only that each cap contain \emph{some} point whose normal is $O(K^{-2})$-close to the plane; it does not require this of every normal in the cap.

\begin{theorem}[Narrow decoupling; \cite{GLMX}]\label{thm:decoupling}
Let $k\ge3$, $r_k=2(k-1)/(k-3)$, and $F=\sum_\tau F_\tau$, where each $\widehat F_\tau$ is supported in an $O(K^{-2})$-neighborhood of a $K^{-1}$-cap. Suppose that there is a $(k-1)$-plane $V$ such that each cap contains a point whose normal lies at distance $O(K^{-2})$ from $V$. Then, for every cube $B$ of side $K^2$, every $\eta>0$, and every $2\le s\le r_k$ with $s<\infty$,
\[
 \norm F{L^s(B)}\lesssim_\eta K^\eta
       \left(\sum_\tau\norm{F_\tau}{L^s(w_B)}^2\right)^{1/2},
\]
where the weight $w_B$ adapted to $B$ satisfies
\[
 w_B(z)\lesssim_{N_0}
      \bigl(1+K^{-2}|z-c_B|\bigr)^{-N_0}.
\]
Here $c_B$ denotes the center of $B$, and the decay order $N_0$ may be chosen sufficiently large in terms of the fixed parameters. When $k=3$ there are only $O(1)$ relevant caps, and the estimate also holds at the endpoint $s=\infty$, with $L^\infty(B)$ on both sides.
\end{theorem}

We keep the weight $w_B$ in the narrow argument and treat its spatial tails in Section~\ref{subsec:decoupling-tails}; we do not assume that $w_B$ has compact support.

\subsection{Lorentz rescaling and spatial localization}\label{subsec:charts}
Fix a cap $\tau$ of angular width $K^{-1}$. The Lorentz rescaling associated with $\tau$ dilates the cap to angular width comparable to one and contracts the physical variables by the factors $K^{-2}$, $K^{-1}$, and $1$. We first give the exact formulas and then localize at the smaller scale $R/K^2$.

After a fixed rotation, we may work in a region where $\rho=|\xi|+\xi_1$ is bounded away from zero. Write $\xi=(\xi_1,\eta)$ and $x=(x_1,y)$, and set
\[
 s=t-x_1,\qquad v=\frac{t+x_1}{2}.
\]
Since $\xi_1=\rho/2-|\eta|^2/(2\rho)$, the phase becomes
\[
 x\cdot\xi+t|\xi|
    =s\frac{|\eta|^2}{2\rho}+y\cdot\eta+v\rho.
\]
We change variables from $\xi$ to $(\eta,\rho)$, absorb the smooth Jacobian factor into $f$, and continue to write $E$ for the resulting operator:
\begin{equation}\label{eq:null-operator}
 Ef(s,y,v)=\iint e^{i(s|\eta|^2/(2\rho)+y\cdot\eta+v\rho)}
                          f(\eta,\rho)\dd\eta\dd\rho.
\end{equation}
Here $\rho$ ranges over a fixed compact subinterval of $(0,\infty)$ and $\eta/\rho$ over a fixed bounded set. Finitely many such charts cover the truncated cone, and passing to a chart changes the $L^2$ norm, the weight condition, and the estimates used below only by fixed factors.

Suppose that $\tau$ is centered at $w_\tau$ in the variable $\eta/\rho$, and put
\[
 \eta=\rho w_\tau+K^{-1}\zeta,\qquad
 g(\zeta,\rho)=K^{-(n-1)/2}
                f_\tau(\rho w_\tau+K^{-1}\zeta,\rho).
\]
Then $\norm g2=\norm{f_\tau}2$ and
\[
 Ef_\tau(z)=K^{-(n-1)/2}Eg(L_\tau z),
\]
where
\begin{equation}\label{eq:L}
 L_\tau(s,y,v)=
 \left(\frac{s}{K^2},\frac{y+sw_\tau}{K},
       v+y\cdot w_\tau+\frac{s|w_\tau|^2}{2}\right).
\end{equation}
Thus $|\det L_\tau|=K^{-(n+1)}$, and for every measurable set $U$,
\begin{equation}\label{eq:Lnorm}
 \norm{Ef_\tau}{L^{r_k}(U)}
   =K^{\frac{n+1}{r_k}-\frac{n-1}{2}}\norm{Eg}{L^{r_k}(L_\tau U)}.
\end{equation}
The variable $\rho$ is unchanged by the rescaling. In particular, $L_\tau(B_R)$ need not fit inside a single ball of radius $R/K^2$: its extent can be as large as $R/K$ in the $y$-directions and $R$ in the $v$-direction.

To localize, we partition the rescaled function smoothly in its initial spatial variables at length scale $RK^{-2+\delta_1}$, project the pieces back to a slightly enlarged cap, and pull them back to the original variables. The standard localization argument for the cone \cite[Proposition 3.1]{HarrisDecay} gives
\begin{equation}\label{eq:bessel-raw}
 f_\tau=\sum_D f_{\tau,D},\qquad
 \sum_{\tau,D}\norm{f_{\tau,D}}2^2\lesssim\norm f2^2.
\end{equation}
The extensions $Ef_{\tau,D}$ are concentrated on boxes $X_{\tau,D}$ whose axes have lengths
\[
 R,\qquad (R/K)K^{\delta_1}\quad(n-1\text{ axes}),\qquad
 (R/K^2)K^{\delta_1}.
\]
For fixed $\tau$, these boxes have bounded overlap, and for every $z$ in a fixed enlargement of $B_R$,
\begin{equation}\label{eq:packet-tail}
 \sum_{D:z\notin X_{\tau,D}}|Ef_{\tau,D}(z)|
       \lesssim_L R^{-L}\norm{f_\tau}2.
\end{equation}
Indeed, \eqref{eq:bessel-raw} follows from Plancherel's theorem, and \eqref{eq:packet-tail} follows by integration by parts, because the group velocity is bounded over the rescaled time interval of length $O(R/K^2)$. The padding factor $K^{\delta_1}$ allows decay of arbitrary order in $R$.

Throughout the iteration, the radial frequency variable stays in a fixed compact subinterval of $(0,\infty)$. For example, at scale $R$ we may allow the interval
\[
 I_R=[a_*-c_*R^{-\delta_3},\ b_*+c_*R^{-\delta_3}],
 \qquad a_*>2c_*>0,\quad 0<\delta_3<1/10.
\]
The projection used in the localization equals one on $I_R$ and is supported in $I_{R/K^2}$. The margin $(R/K^2)^{-\delta_3}-R^{-\delta_3}$ exceeds the reciprocal of the localization length, and the derivatives of the corresponding cutoff are taken into account in the integration by parts. The angular margins are fixed before the rescaling. In this way we avoid enlarging the radial annulus or the angular chart repeatedly.

\subsection{The weighted bilinear estimate and its normalization}\label{subsec:weighted-bilinear}
For a compactly supported nonnegative measure $\mu$ and $0<a\le n+1$, write
\[
 [\mu]_a=\sup_{z\in\R^{n+1},\,r>0}r^{-a}\mu(B(z,r)).
\]
Unlike \eqref{eq:weight}, this quantity involves \emph{all} radii. To state the bilinear estimate at high frequency, let
\[
 \Gamma_S=\{(\xi,\tau)\in\R^n\times\R:|\xi|=\tau,\ S\le\tau\le2S\},
 \qquad
 \Gamma_S(1)=\{z:\dist(z,\Gamma_S)<1\}.
\]
For a function $F$ supported away from the origin, define its angular support by
$\operatorname{Asupp}F=\{z/|z|:z\in\supp F\}$.

\begin{theorem}[Cho--Ham--Lee; {\cite[Theorem 2.2]{CHL}}]\label{thm:CHL-original}
Let $n\ge3$, $0<a\le n+1$, and $S\gg1$. Suppose that $\mu$ is supported in $B(0,1)$ with $[\mu]_a<\infty$, and that $F_1,F_2\in L^2(\R^{n+1})$ are supported in $\Gamma_S(1)$ with
\[
 \dist(\operatorname{Asupp}F_1,\operatorname{Asupp}F_2)\ge\frac1{100}.
\]
Then, for $2\le p\le\infty$ and
\[
 b>\max\left\{\frac n2-\frac ap,\frac{3n+1-2a}{8}\right\},
\]
one has
\[
 \norm{\widehat F_1\widehat F_2}{L^{p/2}(d\mu)}
 \lesssim_{n,a,p,b}S^{2b}[\mu]_a^{2/p}
                       \norm{F_1}2\norm{F_2}2.
\]
The implicit constant is independent of $S$, $\mu$, $F_1$, and $F_2$. When $p=\infty$, the left-hand side is interpreted as an essential supremum.
\end{theorem}

We shall use the following consequence of Theorem~\ref{thm:CHL-original} at unit frequency scale. For the rest of this subsection, let $n\ge3$, $0<m\le n-1$, and $2\le p<\infty$, and put
\[
 D_{n,m}(p)=\left(\frac mp-\frac{n-1+2m}{8}\right)_+.
\]

\begin{lemma}[Local weighted bilinear estimate]\label{lem:CHL-unit}
Let $S\ge1$, and let $\nu$ be a measure supported in a ball of radius $S$ with $[\nu]_m\le A$. Suppose that the angular supports of $g_1$ and $g_2$ are separated by at least $\omega>0$. Then, for every $\eta>0$,
\[
 \left\||Eg_1Eg_2|^{1/2}\right\|_{L^p(d\nu)}
 \lesssim_{n,m,p,\eta}\omega^{-O(1)}A^{1/p}
       S^{D_{n,m}(p)+\eta}\norm{g_1}2^{1/2}\norm{g_2}2^{1/2}.
\]
The same estimate holds for fixed annuli and cone charts, and for fixed enlargements of the physical ball. The implicit constant and the exponent in $\omega^{-O(1)}$ are independent of $S$, $g_1$, and $g_2$.
\end{lemma}
\begin{proof}
Assume first that the angular separation is bounded below by a fixed constant, and translate the ball containing the support of $\nu$ so that it is centered at the origin. Set $\mu(U)=S^{-m}\nu(SU)$, so that $[\mu]_m=[\nu]_m\le A$. Fix a smooth function $\varphi$, compactly supported near zero, such that $\Phi(T)=\int e^{iTt}\varphi(t)\,dt$ is bounded away from zero for $|T|\le1$, and put
\[
 F_j(\xi,\tau)=S^{-n/2}g_j(\xi/S)\varphi(\tau-|\xi|),
 \qquad j=1,2.
\]
Then $\norm{F_j}2\sim\norm{g_j}2$, and the Fourier integral of $F_j$ with the positive sign in the exponent, evaluated at the point $(X,T)$, equals $S^{n/2}\Phi(T)Eg_j(SX,ST)$. We apply Theorem~\ref{thm:CHL-original}, after reflecting the measure if necessary, and rescale. Apart from the factor $S^\eta$, the resulting power of $S$ is
\[
 \max\left\{\frac n2-\frac mp,\frac{3n+1-2m}{8}\right\}
 -\frac n2+\frac mp=D_{n,m}(p).
\]
When the separation is $\omega$, we subdivide the angular supports into caps of width $c\omega$ and apply the Lorentz rescaling \eqref{eq:L} to each pair of caps. The number of pairs and the changes in the measure constant, in the physical scale, and in the $L^2$ norms cost only a factor $\omega^{-O(1)}$. Finally, when $S$ is bounded, the estimate follows from \eqref{eq:Linfty} and the bound $\nu(B_S)\lesssim AS^m$.
\end{proof}

\subsection{Unweighted multilinear restriction}\label{subsec:unweighted-multilinear}
We first recall two restriction estimates in product form, and then derive the estimate for geometric means that is used below.

\begin{theorem}[Bennett--Carbery--Tao; {\cite[Section 5]{BCT}}]\label{thm:BCT-original}
Let $3\le k\le n+1$. For $1\le j\le k$, let $\Sigma_j:U_j\subset\R^n\to\R^{n+1}$ be a smooth parametrization of a compact hypersurface patch, with fixed $C^2$ bounds. Define
\[
 \mathcal E_jh(z)=\int_{U_j}e^{iz\cdot\Sigma_j(\xi)}h(\xi)\,d\xi,
 \qquad
 Y_j(\xi)=\ast\bigl(\partial_1\Sigma_j(\xi)\wedge\cdots\wedge
                           \partial_n\Sigma_j(\xi)\bigr),
\]
where $\ast$ identifies an $n$-vector with its normal vector in $\R^{n+1}$. Assume that
\[
 |Y_1(\xi_1)\wedge\cdots\wedge Y_k(\xi_k)|\ge\nu>0
 \qquad(\xi_j\in U_j).
\]
If $1\le a\le\infty$, $p\ge2k/(k-1)$, and $a'\le(k-1)p/k$, then for every $\eta>0$ and $S\ge1$,
\[
 \left\|\prod_{j=1}^k\mathcal E_jh_j\right\|_{L^{p/k}(B_S)}
 \le C_\eta S^\eta\prod_{j=1}^k\norm{h_j}{L^a(U_j)}.
\]
Here $a'$ is the H\"older conjugate of $a$, and $C_\eta$ may depend on the patches, on their regularity bounds, on $\nu$, and on the exponents, but not on $S$ or on the functions $h_j$. No curvature hypothesis is imposed.
\end{theorem}

\begin{theorem}[Wolff; {\cite[Theorem 1]{Wolff}}]\label{thm:Wolff-original}
Let $n\ge2$, let $\Gamma=\{(\xi,|\xi|):1\le|\xi|\le2\}\subset\R^{n+1}$, and let $d\sigma$ be its surface measure. For disjoint closed sets $\Omega_1,\Omega_2\subset\Sphere^{n-1}$, put
\[
 \Gamma_j=\{(\xi,|\xi|)\in\Gamma:\xi/|\xi|\in\Omega_j\},
 \qquad j=1,2.
\]
If $h_j\in L^2(d\sigma)$ is supported in $\Gamma_j$, then for every $p>(n+3)/(n+1)$,
\[
 \norm{\widehat{h_1\,d\sigma}\,\widehat{h_2\,d\sigma}}{L^p(\R^{n+1})}
 \lesssim_{n,p,\Gamma_1,\Gamma_2}
              \norm{h_1}{L^2(d\sigma)}\norm{h_2}{L^2(d\sigma)}.
\]
\end{theorem}

The next lemma specializes these theorems to the extension operator $E$ for the cone. In the parametrization $\xi\mapsto(\xi,|\xi|)$ one has $d\sigma=\sqrt2\,d\xi$, so the formulation in terms of the surface measure and the parametrized formulation differ only by a fixed factor and by the sign in the Fourier transform.

\begin{lemma}\label{lem:unweighted-multilinear}
Let $n\ge2$ and $3\le k\le n+1$, and suppose that $g_1,\ldots,g_k$ are $\nu$-transverse. Then, for every $\eta>0$ and $S\ge1$,
\begin{equation}\label{eq:unweighted-multilinear-endpoint}
 \left\|\prod_{i=1}^k|Eg_i|^{1/k}\right\|_{L^{v_k}(B_S)}
 \lesssim_\eta \nu^{-O(1)}S^\eta
                 \prod_{i=1}^k\norm{g_i}2^{1/k}.
\end{equation}
The estimate remains valid in fixed frequency charts and annuli. The power of $\nu^{-1}$ may depend on $\eta$.
\end{lemma}
\begin{proof}
Apply Theorem~\ref{thm:BCT-original} with $a=2$ and $p=2k/(k-1)$, take the $k$-th root, and relabel the error exponent. This proves \eqref{eq:unweighted-multilinear-endpoint} with $2k/(k-1)$ in place of $v_k$. Next, Theorem~\ref{thm:Wolff-original} gives a global bilinear estimate for the geometric mean at every exponent strictly larger than $2(n+3)/(n+1)$. Choosing such an exponent sufficiently close to $2(n+3)/(n+1)$ and applying H\"older's inequality on $B_S$, we obtain, for angularly separated $g$ and $h$,
\[
 \left\||EgEh|^{1/2}\right\|_{L^{2(n+3)/(n+1)}(B_S)}
 \lesssim_\eta\nu^{-O(1)}S^\eta\norm g2^{1/2}\norm h2^{1/2}.
\]
Thus we use only the local estimate at the endpoint exponent, with an $S^\eta$ loss. With the convention $g_{k+1}=g_1$, cyclic pairing gives
\[
 \prod_{i=1}^k|Eg_i|^{1/k}
 =\prod_{i=1}^k\left(|Eg_iEg_{i+1}|^{1/2}\right)^{1/k}.
\]
Applying H\"older's inequality to these $k$ bilinear factors proves the estimate at the exponent $2(n+3)/(n+1)$. Since $v_k$ is the smaller of the two exponents, this proves \eqref{eq:unweighted-multilinear-endpoint}. To obtain the quantitative dependence on $\nu$, we subdivide the supports into sufficiently small caps and normalize the transverse directions. The number of pieces and the changes in the physical scale and in the $L^2$ norms are polynomial in $\nu^{-1}$. For the bilinear bound, we use the same Lorentz rescaling as in the proof of Lemma~\ref{lem:CHL-unit}.

Finally, interpolation with $\norm{Eg_i}\infty\lesssim\norm{g_i}2$ gives the same bound, with an $S^\eta$ loss, at every finite exponent greater than or equal to $v_k$, and the pointwise bound at the exponent $\infty$.
\end{proof}

\subsection{Local packet energies and their two summation bounds}\label{subsec:local-packets}
We first state the multilinear Kakeya estimate for tubes of unit radius.

\begin{theorem}[Multilinear Kakeya; {\cite{BCT}}]\label{thm:DGLZ-Kakeya}
Let $2\le k\le n+1$, and let $\Omega_i\subset\Sphere^n$, $1\le i\le k$, satisfy
\[
 |e_1\wedge\cdots\wedge e_k|\ge\nu>0
 \qquad(e_i\in\Omega_i).
\]
For each $i$, let $\ell_{i,1},\ldots,\ell_{i,N_i}$ be lines in $\R^{n+1}$ with directions in $\Omega_i$, and let $T_{i,a}$ be the unit neighborhood of $\ell_{i,a}$. For every cube $Q_S$ of side length $S\ge1$ and every $\eta>0$,
\[
 \int_{Q_S}\prod_{i=1}^k
       \left(\sum_{a=1}^{N_i}\mathbf1_{T_{i,a}}(z)\right)^{1/(k-1)}\,dz
 \lesssim_\eta\nu^{-O(1)}S^\eta
                     \prod_{i=1}^kN_i^{1/(k-1)}.
\]
The implicit constant is independent of $S$, of the positions of the lines, and of the cardinalities $N_i$.
\end{theorem}

By repeating lines and using homogeneity and approximation, we can extend Theorem~\ref{thm:DGLZ-Kakeya} to sums with nonnegative coefficients. After a dilation, we obtain the following: if each $\mathcal T_i$ is a finite family of $\rho$-neighborhoods $T$ of lines with directions in $\Omega_i$, where $\rho\le S$, and if $c_{i,T}\ge0$, then
\begin{equation}\label{eq:kakeya-weighted-scaled}
 \int_{Q_S}\prod_{i=1}^k
       \left(\sum_{T\in\mathcal T_i}c_{i,T}\mathbf1_T(z)\right)^{1/(k-1)}\,dz
 \lesssim_\eta\nu^{-O(1)}\rho^{n+1}(S/\rho)^\eta
       \prod_{i=1}^k\left(\sum_{T\in\mathcal T_i}c_{i,T}\right)^{1/(k-1)}.
\end{equation}
The same bound holds for tubes of finite length, because each such tube is contained in the corresponding neighborhood of the line through its axis.

For the rest of this subsection, let $n\ge3$ and $3\le k\le n+1$, and let $\mathcal Q$ be the family of lattice $R^{1/2}$-cubes meeting $B_R$. Suppose that $f_1,\ldots,f_k$ are $\nu$-transverse with $\nu\ge K^{-O(1)}$. The angular subdivisions used in the broad selection (Section~\ref{subsec:broad-selection}) are fixed before $\delta_1$ is chosen sufficiently small.

We use the wave packet decomposition at scale $R$, with frequency cubes of side $R^{-1/2}$ and tubes of length $O(R)$ and radius $O(R^{1/2+\delta_2})$; see \cite[Section 2]{CHL} and \cite{HarrisRS}. Let $\{\psi_\tau\}$ and $\{\chi_v\}$ be smooth partitions of unity with bounded overlap, adapted respectively to frequency cubes $\tau$ of side $R^{-1/2}$ and to spatial cubes of side $R^{1/2}$ centered at the points $v\in R^{1/2}\mathbb Z^n$. Choose smooth functions $\widetilde\psi_\tau$ that equal one on $\supp\psi_\tau$ and are supported in a fixed dilate of $\tau$. For the tube $T=T_{\tau,v}$, define
\[
 f_{i,T}=\widetilde\psi_\tau\,\mathcal F
       \bigl(\chi_v\,\mathcal F^{-1}(\psi_\tau f_i)\bigr),\qquad
 c_{i,T}:=\bigl\|\chi_v\,\mathcal F^{-1}(\psi_\tau f_i)
                         \bigr\|_{L^2(\R^n)}^2,
\]
where $\mathcal F$ denotes the unitary Fourier transform. Thus $c_{i,T}$ is the squared $L^2$ norm of the packet before the final frequency cutoff $\widetilde\psi_\tau$ is applied. Plancherel's theorem and bounded overlap give
\[
 \sum_Tc_{i,T}\lesssim\norm{f_i}2^2,\qquad
 \left\|\sum_{T\in\mathcal T'}f_{i,T}\right\|_2^2
       \lesssim\sum_{T\in\mathcal T'}c_{i,T}
\]
for every subcollection $\mathcal T'$ of packets. For the second inequality, we use the bounded overlap of the supports of the functions $\widetilde\psi_\tau$, and then Plancherel's theorem and the bounded overlap of the spatial cutoffs $\chi_v$ for each fixed $\tau$.

Each $Ef_{i,T}$ decays rapidly away from $T$. By the Cauchy--Schwarz inequality and the decay of the localization kernels, the packets with distant initial centers contribute $\Rap(R)\norm{f_i}2$ on $B_R$. The decomposition enlarges the frequency supports by $O(R^{-1/2})$, which preserves the $k$-fold transversality when $\delta_1$ is sufficiently small. Let $\mathcal T_i(Q)$ be the set of packets whose tubes meet a fixed enlargement of $Q$, and set
\[
 f_{i,Q}=\sum_{T\in\mathcal T_i(Q)}f_{i,T},\qquad
 \mathcal E_i(Q)=\sum_{T\in\mathcal T_i(Q)}c_{i,T},\qquad
 G_Q=\prod_{i=1}^k\mathcal E_i(Q)^{1/(2k)}.
\]
On $Q$ we have
\begin{equation}\label{eq:local-energy-approx}
 Ef_i=Ef_{i,Q}+\Rap(R)\norm{f_i}2,\qquad
 \norm{f_{i,Q}}2^2\lesssim\mathcal E_i(Q).
\end{equation}

\begin{lemma}[Two bounds for the local energies]\label{lem:energy-summation}
For every $\eta>0$,
\begin{align}
 \|G_Q\|_{\ell^2(\mathcal Q)}
 &\lesssim R^{1/4+O(\delta_2)}\prod_{i=1}^k\norm{f_i}2^{1/k},
       \label{eq:energy-l2}\\
 \|G_Q\|_{\ell^{2k/(k-1)}(\mathcal Q)}
 &\lesssim_\eta\nu^{-O(1)}R^{\eta+O(\delta_2)}
                    \prod_{i=1}^k\norm{f_i}2^{1/k}.
       \label{eq:energy-l3}
\end{align}
Consequently, for $2\le p<\infty$,
\begin{equation}\label{eq:energy-lp}
 \|G_Q\|_{\ell^p(\mathcal Q)}
 \lesssim_\eta\nu^{-O(1)}
R^{(\frac{k}{2p}-\frac{k-1}{4})_++\eta+O(\delta_2)}
 \prod_{i=1}^k\norm{f_i}2^{1/k}.
\end{equation}
\end{lemma}
\begin{proof}
A tube of length $O(R)$ and radius $O(R^{1/2+\delta_2})$ meets at most $R^{1/2+O(\delta_2)}$ cubes of side $R^{1/2}$. Hence
\[
 \sum_Q\mathcal E_i(Q)\lesssim
 R^{1/2+O(\delta_2)}\sum_Tc_{i,T}
 \lesssim R^{1/2+O(\delta_2)}\norm{f_i}2^2.
\]
By H\"older's inequality,
\[
 \sum_QG_Q^2\le
 \prod_{i=1}^k\left(\sum_Q\mathcal E_i(Q)\right)^{1/k},
\]
which proves \eqref{eq:energy-l2}. In this count, a packet is counted once for each cube that its tube meets.

For the second bound, enlarge each tube $T$ by a fixed factor to a tube $T^*$ such that $Q\subset T^*$ whenever $T\in\mathcal T_i(Q)$, and put $F_i=\sum_Tc_{i,T}\mathbf1_{T^*}$. We apply \eqref{eq:kakeya-weighted-scaled}, the weighted and rescaled form of Theorem~\ref{thm:DGLZ-Kakeya}, with $\rho\sim R^{1/2+\delta_2}$ and $S\sim R$. Since $\rho^{n+1}=R^{(n+1)/2+O(\delta_2)}$, we obtain
\[
\begin{aligned}
 R^{(n+1)/2}\sum_Q\prod_{i=1}^k\mathcal E_i(Q)^{1/(k-1)}
 &\le\int_{B_{O(R)}}\prod_{i=1}^k F_i(z)^{1/(k-1)}\,dz\\
 &\lesssim_\eta\nu^{-O(1)}R^{(n+1)/2+\eta+O(\delta_2)}
            \prod_{i=1}^k\left(\sum_Tc_{i,T}\right)^{1/(k-1)}.
\end{aligned}
\]
Dividing by $R^{(n+1)/2}$, raising both sides to the power $(k-1)/(2k)$, and relabeling $\eta$, we obtain \eqref{eq:energy-l3}.

For $2\le p\le2k/(k-1)$, we interpolate between \eqref{eq:energy-l2} and \eqref{eq:energy-l3}. The weight on the $\ell^2$ estimate is $2k/p-(k-1)$, which gives the power $k/(2p)-(k-1)/4$ of $R$. For larger $p$, we use the monotonicity of the $\ell^p$ norms. This proves \eqref{eq:energy-lp}.
\end{proof}

\section{The discrete induction and the reduction to weights}\label{sec:discrete}
It suffices to prove Theorem~\ref{thm:main} for $m>0$. Indeed, a weight that satisfies \eqref{eq:weight} with $m=0$ also satisfies it with every $m'\in(0,n-1]$ in place of $m$, with the same constant $C_H$. The three exponents in the theorem are continuous in $m$ at $m=0$, and the admissible ranges of $p$ do not depend on $m$. Given $\varepsilon>0$, we fix $m'>0$, independent of $R$, so small that each exponent increases by at most $\varepsilon/2$, and we apply the estimate for $m'$ with the loss $R^{\varepsilon/2}$. This gives the claimed bound at $m=0$; the argument does not require the constants to be uniform as $m'\downarrow0$. The assertion without an $R^\varepsilon$ loss for $p\ge q$ uses only the ball condition at unit scale and is proved in Section~\ref{subsec:endpoint}.

\subsection{The discrete proposition}
Let $Y=\bigcup_{B\in\cB}B$ be a union of $M$ lattice cubes of side $K^2$ contained in $B_R$. We define its \emph{density} by
\[
 \gamma(Y)=\sup_{\substack{z\in\R^{n+1}\\K^2\le\rho\le R^{1/2}}}
 \frac{\#\{B\in\cB:B\cap B(z,\rho)\ne\varnothing\}}{\rho^m}.
\]
Since the definition counts the cubes that meet $B(z,\rho)$, it can be used with cubes from different grids; counting instead the cubes contained in a fixed enlargement of $B(z,\rho)$ gives an equivalent quantity. Enlarging the balls at the endpoint scale $\rho=R^{1/2}$ by a fixed factor costs only a constant, because a ball of radius $O(\sqrt R)$ is covered by finitely many balls of radius $\sqrt R$.

\begin{proposition}\label{prop:discrete}
Let $n$, $m$, $k$, $p$, and $\beta$ be as in one of the parts \textup{(i)}--\textup{(iii)} of Theorem~\ref{thm:main}, with $m>0$, and let $r_k$ be given by \eqref{eq:restriction-exponents}.
For every $\varepsilon>0$, one can choose $\delta_1>0$ sufficiently small so that the following holds for every $R\ge1$, with $K=R^{\delta_1}$. Suppose that $Y$ is a union of $M$ lattice $K^2$-cubes $B$ as above, and that
\[
 \Lambda\le\norm{Ef}{L^{r_k}(B)}<2\Lambda
 \qquad(B\subset Y)
\]
for some $\Lambda>0$. Then
\begin{equation}\label{eq:discrete}
 \norm{Ef}{L^{r_k}(Y)}
 \le C_\varepsilon R^{\beta+\varepsilon}
          \left(\frac{\gamma(Y)}M\right)^{\frac1p-\frac1{r_k}}\norm f2.
\end{equation}
The same assertion holds in the fixed cone charts of Section~\ref{subsec:charts} and for the variants of the lattice allowed in Section~\ref{subsec:conventions}.
\end{proposition}

We prove Proposition~\ref{prop:discrete} by induction on $R$; the narrow case is treated in Section~\ref{sec:narrow} and the three broad bounds in Section~\ref{sec:broad}. The reduction of Theorem~\ref{thm:main} to Proposition~\ref{prop:discrete}, given in Section~\ref{subsec:weighted-reduction}, uses only the condition $2\le p\le r_k$ and the form of \eqref{eq:discrete}, so it applies to all three parts of the theorem.

If $Y\ne\varnothing$, the definition of the density and a cover of $B_R$ by balls of radius $\sqrt R$ give
\begin{equation}\label{eq:elementary-density}
 \gamma(Y)\gtrsim K^{-2m},\qquad M\lesssim\gamma(Y)R^{(n+1+m)/2}.
\end{equation}
The case of small amplitude is immediate from these bounds, and the bound $\Lambda\lesssim K^{2(n+1)/r_k}\norm f2$ always holds by \eqref{eq:Linfty}. We may therefore assume that
\[
 R^{-L}\norm f2\le\Lambda\lesssim K^{2(n+1)/r_k}\norm f2
\]
for a sufficiently large fixed $L$.

\subsection{Reduction to a dyadic weight class}\label{subsec:weighted-reduction}
Assume that Proposition~\ref{prop:discrete} holds with $\varepsilon/2$ in place of $\varepsilon$. Replace $H$ by $H\mathbf1_{B_R}$, and tile a fixed enlargement of $B_R$ by $K^2$-cubes $B$. Put
\[
 h_B=\int_BH,\qquad U_B=\norm{Ef}{L^{r_k}(cB)},
\]
where $c$ is a fixed dilation factor that is large enough for \eqref{eq:localaverage} to apply. The ball condition \eqref{eq:weight} gives
\[
 \int H\lesssim C_HR^{(n+1+m)/2},\qquad h_B\lesssim C_HK^{2m}.
\]
We pigeonhole the dyadic values of $h_B/C_H$ and $U_B/\norm f2$, discarding the classes with negligible contribution, and use a finite coloring to ensure that the dilated cubes $cB$ are pairwise disjoint. Choosing the class that contributes most to $\sum_Bh_BU_B^p$, we obtain a family of $M$ cubes with $h_B\sim h$ and $U_B\sim U$ such that, writing $Y^+=\bigcup cB$, we have
\[
 \norm{Ef}{L^p(B_R,H)}
 \lessapprox h^{1/p}M^{1/p-1/r_k}
                   \norm{Ef}{L^{r_k}(Y^+)}
       +C_H^{1/p}\Rap(R)\norm f2.
\]
The density $\gamma$ of the selected family satisfies $h\gamma\lesssim C_H$. To see this, sum the masses $h_B$ of the pairwise disjoint original cubes $B$ whose dilates $cB$ meet a test ball: these cubes lie in a fixed enlargement of the ball, and at the endpoint scale $\sqrt R$ this enlargement is covered by finitely many $\sqrt R$-balls.

We now apply Proposition~\ref{prop:discrete} to $Y^+$. The factors involving $h$, $M$, and $\gamma$ combine to give
\[
 h^{1/p}M^{1/p-1/r_k}(\gamma/M)^{\frac1p-\frac1{r_k}}
   =(h\gamma)^{\frac1p-\frac1{r_k}} h^{1/r_k} \lesssim C_H^{1/p}K^{2m/r_k}.
\]
This power of $K$ is incurred only once. After decreasing $\delta_1$ if necessary, we can absorb it, together with the logarithmic losses, into the remaining factor $R^{\varepsilon/2}$. This proves \eqref{eq:full}.

\subsection{The unweighted endpoint and interpolation}\label{subsec:endpoint}
To remove the $R^\varepsilon$ loss at $p=q$, we use the global Strichartz estimate for the wave equation at unit frequency,
\[
 \norm{Ef}{L^{q}(\R^{n+1})}\lesssim_n\norm f2
\]
together with local averaging; see, for instance, \cite[(1.1)]{HarrisRS}. Since $\int_B H\le C_H$ for every unit ball $B$, summing over unit cubes gives $\|(H\mathbf1_{B_R})*|\psi|\|_\infty\lesssim C_H$. Applying Jensen's inequality to $Ef=(Ef)*\psi$ and integrating against $H$, we obtain
\[
 \norm{Ef}{L^{q}(B_R,H)}\lesssim C_H^{1/q}\norm f2.
\]
Interpolating with \eqref{eq:Linfty} gives the exponent $\beta=0$ for every $p\ge q$. Finally, since the factors $C_H^{1/p_j}$ interpolate to $C_H^{1/p}$, the estimates \eqref{eq:full} interpolate to \eqref{eq:combined-upper}.

\section{The narrow case}\label{sec:narrow}\label{concludetheproof}
Fix $n\ge2$, $0<m\le n-1$, $3\le k\le n+1$, and $2\le p\le r_k$ with $p<\infty$. We work in the setting of Proposition~\ref{prop:discrete}, but with an arbitrary exponent $\beta\ge0$, and we assume that \eqref{eq:discrete} holds at smaller radii. The narrow step requires only that $\beta\ge N_{n,m,k}(p)$; the broad bounds for the three parts of Theorem~\ref{thm:main} are supplied in Section~\ref{sec:broad}.

\subsection{Broad and narrow cubes}\label{subsec:broad-narrow}
Suppose that $\norm{Ef}{L^{r_k}(B)}\sim\Lambda$ for each of the $M$ cubes $B$ of $Y$, and that $M'$ of these cubes, denoted by $B'$, are retained. Then
\begin{equation}\label{eq:recover}
 \norm{Ef}{L^{r_k}(\bigcup B)}
 \lesssim(M/M')^{1/r_k}\norm{Ef}{L^{r_k}(\bigcup B')}.
\end{equation}
When $r_k=\infty$, the factor $(M/M')^{1/r_k}$ equals one. Accordingly, in a family of cubes with comparable local norms we select subfamilies by their cardinality, whereas in the weighted reduction of Section~\ref{subsec:weighted-reduction} we selected by the contribution to the integral.

As in Section~\ref{subsec:broad-norms}, decompose $f=\sum_\tau f_\tau$ using a smooth partition of unity adapted to boundedly overlapping $K^{-1}$-caps. There are $O(K^{n-1})$ caps, and
\[
 \sum_\tau\norm{f_\tau}2^2\lesssim\norm f2^2.
\]
Fix $A$ large enough for the broad estimates of Section~\ref{sec:broad}. For each cube $B$, choose $(k-1)$-planes $V_1(B),\ldots,V_A(B)$, the exceptional planes of $B$, that attain the minimum
\[
 \min_{V_1,\ldots,V_A}
 \max_{\substack{\tau:\ \angle(G(\tau),V_a)>K^{-2}\\1\le a\le A}}
       \norm{Ef_\tau}{L^{r_k}(B)}.
\]
When $r_k$ is finite, this minimum equals $\mu_{Ef,r_k,A}(B)^{1/r_k}$; when $r_k=\infty$, the norms in the display are local suprema. Assign each cap that lies within angle $K^{-2}$ of some exceptional plane to the first such plane. This gives disjoint families $\mathcal T_a(B)$, $1\le a\le A$, and we put $\mathcal T(B)=\bigcup_{a=1}^A\mathcal T_a(B)$. The triangle inequality gives
\begin{equation}\label{eq:broad-narrow-split}
 \norm{Ef}{L^{r_k}(B)}
 \le \sum_{a=1}^A
       \left\|\sum_{\tau\in\mathcal T_a(B)}Ef_\tau\right\|_{L^{r_k}(B)}\quad+C_nK^{n-1}
       \max_{\tau\notin\mathcal T(B)}\norm{Ef_\tau}{L^{r_k}(B)},
\end{equation}
where $C_nK^{n-1}$ is an upper bound for the number of caps. We call $B$ \emph{narrow} if the first term on the right-hand side is at least as large as the second, and \emph{broad} otherwise. In either case, the left-hand side is at most twice the dominant term. At least $M/2$ of the cubes are of the same type, and by \eqref{eq:recover} it suffices to prove the estimate for that family. The classification is made at the exponent $r_k$ and does not depend on the exponent that is used later in the broad estimates.

\subsection{Decoupling and the spatial tails}\label{subsec:decoupling-tails}
Assume that at least $M/2$ of the cubes are narrow, and let $\mathcal B_{\rm nar}$ be the family of narrow cubes. Using the rescaling of Section~\ref{subsec:charts}, we shall apply the induction hypothesis at scale $R/K^2$, where the lattice cubes have side $(R/K^2)^{2\delta_1}=K^{2-4\delta_1}$. We refer to the $K^2$-cubes at scale $R$ as \emph{old cubes} and to the lattice cubes at scale $R/K^2$ as \emph{new cubes}. The case of bounded radii is absorbed into the induction constant.

On a narrow cube, the local norm is controlled by the first term on the right-hand side of \eqref{eq:broad-narrow-split}. We apply Theorem~\ref{thm:decoupling} to each family $\mathcal T_a(B)$ and then the Cauchy--Schwarz inequality in $a$. When $r_k$ is finite, this gives
\[
 \norm{Ef}{L^{r_k}(B)}
 \lesssim_\eta A^{1/2}K^\eta
 \left(\sum_{\tau\in\mathcal T(B)}
               \norm{Ef_\tau}{L^{r_k}(w_B)}^2\right)^{1/2}.
\]
The functions $f_\tau$ are defined globally and do not depend on $B$. From now on, the fixed factor $A^{1/2}$ is absorbed into the implicit constants. Next, we decompose the adapted weight over the lattice translates $B_\ell=B+K^2\ell$:
\[
 w_B\lesssim_{N_0}\sum_{\ell\in\mathbb Z^{n+1}}
          (1+|\ell|)^{-N_0}\mathbf1_{B_\ell}.
\]
Since $r_k\ge2$, the $\ell^{r_k}$ norm of a sequence indexed by $\ell$ is bounded by its $\ell^2$ norm. Thus
\[
 \norm{Ef}{L^{r_k}(B)}
 \lesssim_\eta K^\eta
 \left(\sum_{\ell\in\mathbb Z^{n+1}}
       (1+|\ell|)^{-2N_0/r_k}
       \sum_{\tau\in\mathcal T(B)}
                 \norm{Ef_\tau}{L^{r_k}(B_\ell)}^2\right)^{1/2}.
\]
By \eqref{eq:Linfty}, the contribution of the terms with $|\ell|>K^{\delta_1}$ is negligible once $N_0$ is chosen sufficiently large in terms of $L$ and $\delta_1$. For each old cube we choose a value of $\ell$ that maximizes the inner sum, and then we pigeonhole a common value $\ell_0$. This gives a family $\mathcal B_{\rm loc}\subseteq\mathcal B_{\rm nar}$ with
$\#\mathcal B_{\rm loc}\gtrsim K^{=O(\delta_1)}M$ such that, writing $B^\sharp=B+K^2\ell_0$, we have for every $B\in\mathcal B_{\rm loc}$
\begin{equation}\label{eq:shifted-decoupling}
 \norm{Ef}{L^{r_k}(B)}
 \lesssim K^{\eta+O(\delta_1)}
       \left(\sum_{\tau\in\mathcal T(B)}
             \norm{Ef_\tau}{L^{r_k}(B^\sharp)}^2\right)^{1/2}
       +\Rap(R)\norm f2.
\end{equation}
When $r_k=\infty$, which happens only for $k=3$, each exceptional plane meets only $O(1)$ relevant caps, so $\mathcal T(B)$ has $O(A)$ members. In this case \eqref{eq:shifted-decoupling} follows directly from the triangle inequality, with $\ell_0=0$ and $\mathcal B_{\rm loc}=\mathcal B_{\rm nar}$.

The translation vector $K^2\ell_0$ is the same for all the retained cubes and has length $O(K^{2+\delta_1})=o(R)$. Hence the family $\{B^\sharp:B\subset Y\}$ has the same density $\gamma$ as the original family and lies in a fixed enlargement of $B_R$. Narrowness and the condition of comparable local norms refer to the original cube $B$; neither is required of $B^\sharp$. In this way the tails of the weight are handled without a separate localized decoupling lemma.

\subsection{Rescaled cells and source windows}
The preimage under $L_\tau$ of a new cube $B$ of side $K^{2-4\delta_1}$ is a \emph{cell} $S=L_\tau^{-1}B$ whose axes have lengths
\[
 K^{4-4\delta_1},\qquad K^{3-4\delta_1}\quad(n-1\text{ axes}),\qquad K^{2-4\delta_1}.
\]
For fixed $\tau$, the cells tile $\R^{n+1}$. Since $K^2/K^{2-4\delta_1}=K^{4\delta_1}$, a fixed enlargement $cB^\sharp$ of a translated old cube meets at most
\begin{equation}\label{eq:J}
 J_K\lesssim K^{4(n+1)\delta_1}
\end{equation}
cells. In particular, an old cube need not be contained in a single cell.

We now use the decomposition \eqref{eq:bessel-raw} and the tail estimate \eqref{eq:packet-tail}, and we refer to the boxes $X_{\tau,D}$ as \emph{source boxes}. The image $L_\tau X_{\tau,D}$ of a source box has diameter $O(RK^{-2+\delta_1})$. We distribute the cells that meet a source box among $K^{O(\delta_1)}$ \emph{windows}, so that the images under $L_\tau$ of the cells in one window lie in a single ball of radius $R/K^2$. We write $i=(\tau,D,j)$ for a source box together with one of its windows, call $i$ a \emph{source}, and set $f_i=f_{\tau,D}$. Each function $f_{\tau,D}$ occurs for at most $K^{O(\delta_1)}$ sources. Hence
\begin{equation}\label{eq:bessel}
 \sum_i\norm{f_i}2^2\lesssim K^{O(\delta_1)}\norm f2^2.
\end{equation}
Let $\Sigma_i$ be the set of cells in the window of $i$ that meet $cB^\sharp$ for some $B\in\mathcal B_{\rm loc}$. Their images under $L_{\tau(i)}$ lie in a ball of radius $R/K^2$.

For $S\in\Sigma_i$, define
\begin{equation}\label{eq:local-piece}
 a_{i,S}(B)=
 \norm{Ef_i}{L^{r_k}(cB^\sharp\cap S\cap X_{\tau(i),D(i)})}
       \mathbf1_{\{\tau(i)\in\mathcal T(B)\}}.
\end{equation}
Using \eqref{eq:shifted-decoupling}, the bounded overlap of the source boxes, and the inequality $\|\cdot\|_{\ell^{r_k}}\le\|\cdot\|_{\ell^2}$ for the sum over the cells, we obtain
\[
 \norm{Ef}{L^{r_k}(B)}
 \lesssim K^{\eta+O(\delta_1)}
       \left(\sum_i\sum_{S\in\Sigma_i}a_{i,S}(B)^2\right)^{1/2}
       +\Rap(R)\norm f2.
\]
The same bound holds when $r_k=\infty$, with the norms interpreted as suprema. All subsequent selections are made among the nonnegative numbers $a_{i,S}(B)$ in this quadratic sum.

\subsection{Dyadic selection}
We discard the sources and the cells whose norms are negligible. For each source $i$, we first sort its cells $S\in\Sigma_i$ according to the dyadic values $\kappa$ of the cell norm and $\eta_1$ of the \emph{occupancy}:
\[
 \norm{Ef_i}{L^{r_k}(S)}\sim\kappa,
 \qquad
 \#\{B\in\mathcal B_{\rm loc}:cB^\sharp\cap S\ne\varnothing\}
       \sim\eta_1.
\]
Let $\Sigma_i(\kappa,\eta_1)$ denote the resulting family of cells. Next, we sort the sources according to the dyadic values of $M_i=\#\Sigma_i(\kappa,\eta_1)$, of $\norm{f_i}2$, and of the density
\[
 \gamma_i=\sup_{\substack{z'\in\R^{n+1}\\K^{2-4\delta_1}\le\rho\le\sqrt R/K}}
 \frac{\#\{S\in\Sigma_i(\kappa,\eta_1):
          L_{\tau(i)}S\cap B(z',\rho)\ne\varnothing\}}{\rho^m}.
\]
We denote these dyadic values by $M_1$, $b$, and $\gamma_1$, respectively, and collect the five parameters in the label $\zeta=(\kappa,\eta_1,M_1,\gamma_1,b)$. The density $\gamma_i$ is computed \emph{after} the cell norm and the occupancy have been fixed.

Let $\mathcal I_\zeta$ be the family of sources with label $\zeta$, and put
\[
 Z_\zeta(B)=\sum_{i\in\mathcal I_\zeta}
                 \sum_{S\in\Sigma_i(\kappa,\eta_1)}a_{i,S}(B)^2.
\]
For each $B\in\mathcal B_{\rm loc}$ we choose a label that maximizes $Z_\zeta(B)$, and then we pigeonhole a label $\zeta_0$ that is chosen for all cubes in a family $\mathcal B_0\subseteq\mathcal B_{\rm loc}$. Thus
\[
 \norm{Ef}{L^{r_k}(B)}
       \lessapprox K^{\eta+O(\delta_1)}
                    Z_{\zeta_0}(B)^{1/2},\qquad B\in\mathcal B_0,
\]
and $\#\mathcal B_0\gtrapprox K^{-O(\delta_1)}M$. From now on we fix $\zeta_0$, write $\mathcal I$ for $\mathcal I_{\zeta_0}$, and write $\Sigma_i$ for the family $\Sigma_i(\kappa,\eta_1)$ of selected cells of a source $i\in\mathcal I$. Then
\begin{equation}\label{eq:uniform-source-data}
 \begin{gathered}
 \#\Sigma_i\sim M_1,\quad \gamma_i\sim\gamma_1,\quad
       \norm{f_i}2\sim b,\\
 \norm{Ef_i}{L^{r_k}(S)}\sim\kappa,\quad
 \#\{B\in\mathcal B_{\rm loc}:cB^\sharp\cap S\ne\varnothing\}
       \sim\eta_1\qquad(S\in\Sigma_i).
 \end{gathered}
\end{equation}

Finally, we pigeonhole the \emph{multiplicity}
\[
 \mu(B)=\#\{(i,S):i\in\mathcal I,\ S\in\Sigma_i,
                         \ cB^\sharp\cap S\ne\varnothing\}.
\]
This quantity counts pairs consisting of a source and a cell, and not merely sources. There is a family $\mathcal B'\subseteq\mathcal B_0$ and a dyadic number $\mu$ such that
\begin{equation}\label{eq:N6-retention}
 M':=\#\mathcal B'\gtrapprox K^{-O(\delta_1)}M,
 \qquad \mu(B)\sim\mu\quad(B\in\mathcal B').
\end{equation}
Put $Y_i=\bigcup_{S\in\Sigma_i}S$. For $B\in\mathcal B'$, we apply H\"older's inequality to the sum over the $\lesssim\mu$ pairs $(i,S)$ that contribute to $Z_{\zeta_0}(B)$, which produces the factor $\mu^{\frac12-\frac1{r_k}}$. We then sum over $B\in\mathcal B'$, using the bounded overlap of the sets $cB^\sharp$ in the form
\[
 \sum_{B\in\mathcal B'}a_{i,S}(B)^{r_k}
             \lesssim\norm{Ef_i}{L^{r_k}(S)}^{r_k}\qquad(r_k<\infty).
\]
Together with \eqref{eq:recover}, this gives
\begin{equation}\label{eq:narrow-after-selection}
 \norm{Ef}{L^{r_k}(Y)}
 \lessapprox K^{\eta+O(\delta_1)}\mu^{\frac12-\frac1{r_k}}
           \left(\sum_{i\in\mathcal I}
                     \norm{Ef_i}{L^{r_k}(Y_i)}^{r_k}\right)^{1/r_k}.
\end{equation}
When $r_k=\infty$, the last sum is replaced by a maximum, and the factor is again $\mu^{\frac12-\frac1{r_k}}$. Note that the functions on the right-hand side are the $f_i$ themselves, to which the estimates at the smaller scale apply.

\subsection{The two incidence counts}
For the first count, consider the triples $(B,i,S)$ with $B\in\mathcal B'$, $i\in\mathcal I$, $S\in\Sigma_i$, and $cB^\sharp\cap S\ne\varnothing$. Counting them in two ways and using \eqref{eq:uniform-source-data} and \eqref{eq:N6-retention}, we obtain
$M\mu\lessapprox K^{O(\delta_1)}
                     \#\mathcal I\,M_1\eta_1$. Moreover, \eqref{eq:J} gives $\mu\lesssim J_K\#\mathcal I$.

For the second count, fix $i\in\mathcal I$ and a ball of radius
$\rho\in[K^{2-4\delta_1},\sqrt R/K]$ that attains the density $\gamma_i$ up to a fixed factor. The preimage of this ball under $L_{\tau(i)}$ meets $\gtrsim\gamma_1\rho^m$ selected cells, and each of these cells meets $\gtrsim\eta_1$ translated old cubes. Hence there are $\gtrsim\eta_1\gamma_1\rho^m$ incidences between these cells and translated old cubes, and each translated old cube is counted at most $J_K$ times.

The preimage has axes of lengths $K^2\rho$, $K\rho$ (in $n-1$ directions), and $\rho$. After thickening it by $O(K^2)$, so that it contains the old cubes in question, we can cover it by $O(K)$ balls of radius $O(K\rho)$, because $K^2\lesssim K^{3-4\delta_1}\le K\rho\le\sqrt R$. The family of translated old cubes has density at most $\gamma$, so the number of distinct old cubes involved is $\lesssim\gamma K^{m+1}\rho^m$. Hence
$\eta_1\gamma_1\lesssim J_KK^{m+1}\gamma$.
As before, enlargements by a fixed factor at the endpoint scale are handled by finite covers by $\sqrt R$-balls; the density condition is never used at scales larger than $\sqrt R$. Combining the two counts, we obtain
\begin{equation}\label{eq:muratio}
 \frac\mu{\#\mathcal I}
 \lessapprox K^{O(\delta_1)}
        \min\left\{1,\ K^{m+1}\frac{\gamma M_1}{M\gamma_1}\right\}.
\end{equation}
Here the factor $K^{m+1}$ comes from the Lorentz rescaling, whereas the selection of the translate in Section~\ref{subsec:decoupling-tails} and the mismatch between the grids contribute the factor $K^{O(\delta_1)}$.

\subsection{The smaller-scale estimate and the narrow step}\label{subsec:narrow-step}
Let $g_i$ be the function obtained from $f_i$ by the Lorentz rescaling of Section~\ref{subsec:charts}. Then $\norm{g_i}2=\norm{f_i}2$. Moreover,
$L_{\tau(i)}Y_i$ is a union of $\sim M_1$ new cubes contained in a ball of radius $R/K^2$, it has density $\sim\gamma_1$, and, by \eqref{eq:uniform-source-data} and \eqref{eq:Lnorm}, the local norms of $Eg_i$ on these cubes are comparable, as Proposition~\ref{prop:discrete} requires. The induction hypothesis therefore gives
\[
 \norm{Ef_i}{L^{r_k}(Y_i)}
   \lesssim C_\varepsilon K^{\frac{n+1}{r_k}-\frac{n-1}{2}}
        (\gamma_1/M_1)^{\frac1p-\frac1{r_k}}(R/K^2)^{\beta+\varepsilon}\norm{f_i}2.
\]
Since the norms $\norm{f_i}2$ are comparable to one another, \eqref{eq:bessel} implies
\[
 \left(\sum_{i\in\mathcal I}\norm{f_i}2^{r_k}\right)^{1/r_k}
 \lesssim K^{O(\delta_1)}(\#\mathcal I)^{\frac1{r_k}-\frac12}\norm f2,
\]
with the usual interpretation as a maximum when $r_k=\infty$. Substituting these bounds into \eqref{eq:narrow-after-selection}, we obtain
\[
 \norm{Ef}{L^{r_k}(Y)}\lessapprox
 C_\varepsilon K^{\frac{n+1}{r_k}-\frac{n-1}{2}+\eta+O(\delta_1)}
       \left(\frac\mu{\#\mathcal I}\right)^{\frac12-\frac1{r_k}}
       \left(\frac{\gamma_1}{M_1}\right)^{\frac1p-\frac1{r_k}} (R/K^2)^{\beta+\varepsilon}\norm f2.
\]
Since $0\le\frac1p-\frac1{r_k}\le\frac12-\frac1{r_k}$, \eqref{eq:muratio} gives
\[
 \left(\frac\mu{\#\mathcal I}\right)^{\frac12-\frac1{r_k}}
 \left(\frac{\gamma_1}{M_1}\right)^{\frac1p-\frac1{r_k}}
 \lessapprox K^{(m+1)(\frac1p-\frac1{r_k})+O(\delta_1)}(\gamma/M)^{\frac1p-\frac1{r_k}}.
\]
After the right-hand side of \eqref{eq:discrete} is factored out, the remaining coefficient is bounded by
\begin{equation}\label{eq:close}
 C_{A,\eta}(\log R)^{O(1)}
 K^{2N_{n,m,k}(p)-2\beta-2\varepsilon+\eta+O(\delta_1)},
\end{equation}
because
\[
 \begin{aligned}
 &\frac{n+1}{r_k}-\frac{n-1}{2}
       +(m+1)\left(\frac1p-\frac1{r_k}\right)=\frac{n-m}{r_k}-\frac{n-1}{2}+\frac{m+1}{p}\\
 &\qquad=1-\frac{n-m}{k-1}-(m+1)\left(\frac12-\frac1p\right)
       =2N_{n,m,k}(p).
 \end{aligned}
\]
Now choose $\eta\ll\varepsilon$ and then $\delta_1$ sufficiently small. If $\beta\ge N_{n,m,k}(p)$, then the power of $K=R^{\delta_1}$ in \eqref{eq:close} is negative; it absorbs the logarithmic loss and the fixed factor $C_{A,\eta}$, so that the coefficient is less than $1/2$ once $R$ exceeds a sufficiently large base radius. Bounded radii and negligible errors are then absorbed into $C_\varepsilon$. This proves the narrow step for every $2\le p\le r_k$. It remains to establish the broad bounds, which is done in Section~\ref{sec:broad}.

\begin{remark}\label{rem:no-masked-decoupling}
Decoupling is applied before any of the spatial selections are made. The indicator functions of the sets $cB^\sharp\cap S\cap X_{\tau,D}$ enter only through the nonnegative local norms \eqref{eq:local-piece}. In particular, the argument never requires information about the Fourier support of an oscillatory sum that has been multiplied by such an indicator function.
\end{remark}

\section{The broad case}\label{sec:broad}\label{section-algorithm}\label{sec:bilinear-refinement}
Assume that at least $M/2$ of the cubes are broad in the sense of Section~\ref{subsec:broad-narrow}. We first apply the broad norm estimate \eqref{eq:OW-interpolated} to prove Theorem~\ref{thm:main}\textup{(ii)}, and then use a transverse selection to prove parts~\textup{(i)} and~\textup{(iii)}.

\subsection{Application of the broad norm estimate}\label{subsec:OW-broad}
In this subsection we prove Theorem~\ref{thm:main}\textup{(ii)}. Let $\mathcal B_{\rm br}$ be the family of broad cubes, which has at least $M/2$ members, and let $2\le s\le r_k$ with $s<\infty$. Since the second term dominates in \eqref{eq:broad-narrow-split} on a broad cube, the local averaging estimate \eqref{eq:localaverage} gives
\[
 \Lambda\lesssim K^{O(1)}\mu_{Ef,s,A}(cB)^{1/s},
 \qquad B\in\mathcal B_{\rm br},
\]
where $c$ is a fixed constant and the negligible error term has been absorbed using the lower bound on $\Lambda$. Summing the $s$-th powers over $B\in\mathcal B_{\rm br}$ yields
\begin{equation}\label{eq:OW-recovery}
 \norm{Ef}{L^{r_k}(Y)}
 \lesssim K^{O(1)}M^{1/r_k-1/s}
 \left(\sum_{B\in\mathcal B_{\rm br}}\mu_{Ef,s,A}(cB)\right)^{1/s}.
\end{equation}
Since the dilated cubes $cB$ have bounded overlap, the sum on the right-hand side is controlled by \eqref{eq:OW-interpolated}.

First take $s=p$. Since $\gamma\gtrsim K^{-2m}$ and $1/p-1/r_k\ge0$, \eqref{eq:OW-recovery} gives
\begin{equation}\label{eq:OW-interpolation-bound}
 \norm{Ef}{L^{r_k}(Y)}
 \lesssim_\eta K^{O(1)}
 R^{\frac{n+k+1}{2}(\frac1p-\frac1{\bar p_k})_++\eta}
 (\gamma/M)^{1/p-1/r_k}\norm f2.
\end{equation}
For $2\le p\le\bar p_k$, we may instead take $s=\bar p_k$. The identity
\[
 M^{1/r_k-1/\bar p_k}
 = (\gamma/M)^{1/p-1/r_k}
   \gamma^{1/r_k-1/\bar p_k}
   (M/\gamma)^{1/p-1/\bar p_k}
\]
keeps track of the density and of the cardinality separately. Since $r_k\ge\bar p_k$, the middle factor is at most a fixed power of $K$, and by \eqref{eq:elementary-density} the last factor is at most
\[
 R^{\frac{n+1+m}{2}(\frac1p-\frac1{\bar p_k})}.
\]
Consequently,
\[
 \norm{Ef}{L^{r_k}(Y)}
 \lesssim_\eta K^{O(1)}
 R^{\frac{n+1+m}{2}(\frac1p-\frac1{\bar p_k})+\eta}
 (\gamma/M)^{1/p-1/r_k}\norm f2.
\]
For $2\le p\le\bar p_k$ we take the smaller of this exponent and the one in \eqref{eq:OW-interpolation-bound}, and for $p\ge\bar p_k$ we use \eqref{eq:OW-interpolation-bound} alone. This proves
\[
 \norm{Ef}{L^{r_k}(Y)}
 \lesssim_\eta K^{O(1)}
 R^{\frac{n+1+\min\{m,k\}}2(\frac1p-\frac1{\bar p_k})_++\eta}
 (\gamma/M)^{1/p-1/r_k}\norm f2.
\]
Choose $\eta<\varepsilon/4$, and then choose $\delta_1$ sufficiently small to accommodate \eqref{eq:OW-interpolated}, the one-time losses $K^{O(1)}$, and the requirements of the narrow step. The broad estimate then holds with room to spare in the error exponent, and Section~\ref{subsec:narrow-step} closes the induction whenever $\beta\ge N_{n,m,k}(p)$. By the reduction in Section~\ref{subsec:weighted-reduction}, this proves part~\textup{(ii)}. The only information about $M/\gamma$ used in this argument is the half-scale bound \eqref{eq:elementary-density}.

\subsection{Transverse selection and the linear reduction}\label{subsec:broad-selection}\label{subsec:broad-transfer}
On a broad cube $B$ one can choose $k$ caps, and then subcaps of width $K^{-O(1)}$ with a sufficiently large dimensional exponent, such that
\[
 |G(\xi_1)\wedge\cdots\wedge G(\xi_k)|\ge K^{-O(1)}
\]
for all $\xi_1,\ldots,\xi_k$ in the respective subcaps, and such that the extension of the part of $f$ on each selected subcap has local $L^{r_k}$ norm at least $K^{-O(1)}\norm{Ef}{L^{r_k}(B)}$.

To see this, note that by \eqref{eq:broad-narrow-split} and the minimizing property of the exceptional planes, every $A$-tuple of $(k-1)$-planes leaves a cap, at angle greater than $K^{-2}$ from all of them, whose local $L^{r_k}$ norm is $\gtrsim K^{-(n-1)}\Lambda$. At each step, we extend the span of the directions chosen so far to a $(k-1)$-plane and take the $A$-tuple that consists of $A$ copies of this plane. We then select a cap outside this plane, subdivide it, and retain a subcap with the largest local norm. After $k$ steps, the representative directions have wedge product at least $K^{-O(1)}$. If the subdivision is sufficiently fine, this bound persists throughout the subcaps, and the subdivision costs only a factor $K^{O(1)}$ in the local norms.

This fine subdivision is used only in the broad case; the narrow argument uses only the functions $f_\tau$ associated with the original $K^{-1}$-caps.

Fix $2\le s\le r_k$ with $s<\infty$. For each broad cube $B$, let $h_{1,B},\ldots,h_{k,B}$ be the parts of $f$ on the selected subcaps, and apply Lemma~\ref{lem:translations} to them, with $s$ in place of $p$. This gives translations $v_j(B)$ with $|v_j(B)|\lesssim K^2$ such that
\[
 \Lambda\lesssim K^{O(1)}
 \left\|\prod_{j=1}^k
       |Eh_{j,B}(\,\cdot+v_j(B))|^{1/k}\right\|_{L^s(B)}.
\]
We pigeonhole a subfamily $\mathcal B_2$ of the broad cubes on which the selected subcaps and the translations are the same, and we absorb the translations into the phases, as explained after Lemma~\ref{lem:translations}. This gives $\nu$-transverse functions $f_1,\ldots,f_k$, with $\nu\ge K^{-O(1)}$, such that
\[
 \#\mathcal B_2\gtrapprox K^{-O(1)}M,
 \qquad \norm{f_j}2\lesssim\norm f2\quad(1\le j\le k).
\]
Since the local norms $\norm{Ef}{L^{r_k}(B)}$ are comparable for all the cubes of $Y$, it follows that, for every nonempty $\mathcal B'\subseteq\mathcal B_2$,
\begin{equation}\label{eq:broad-recovery}
 \norm{Ef}{L^{r_k}(Y)}
 \lesssim K^{O(1)}M^{1/r_k}(\#\mathcal B')^{-1/s}
       \left\|\prod_{j=1}^k|Ef_j|^{1/k}\right\|_{L^s(\bigcup\mathcal B')}.
\end{equation}

The following proposition shows how a multilinear estimate can be combined with the narrow step.
\begin{proposition}\label{prop:broad-transfer}
Fix $n\ge2$, $0<m\le n-1$, $3\le k\le n+1$, and $2\le p\le r_k$ with $p<\infty$, and let $b$ be a real number. Suppose that, for every $\eta>0$ and every union $Y$ of lattice $K^2$-cubes in $B_R$, with density $\gamma=\gamma(Y)$, the estimate
\begin{equation}\label{eq:broad-transfer-hypothesis}
 \left\|\prod_{i=1}^k|Ef_i|^{1/k}\right\|_{L^p(Y)}
 \lesssim_\eta K^{O(1)}\gamma^{1/p-1/r_k}
 R^{b+\eta}\prod_{i=1}^k\norm{f_i}2^{1/k}
\end{equation}
holds whenever $f_1,\ldots,f_k$ are $\nu$-transverse with $\nu\ge K^{-O(1)}$. Then \eqref{eq:discrete} and the corresponding weighted estimate \eqref{eq:full} hold with $\beta=\max\{b,N_{n,m,k}(p),0\}$. 
Here the powers of $K$ in \eqref{eq:broad-transfer-hypothesis} are to be fixed before $\delta_1$ is chosen.
\end{proposition}
\begin{proof}
We argue by induction on $R$, with the stated exponent $\beta$. The narrow step is proved in Section~\ref{sec:narrow}. In the broad case, we apply \eqref{eq:broad-recovery} with $s=p$ and $\mathcal B'=\mathcal B_2$, and put $Y'=\bigcup\mathcal B_2$. The lower bound on $\#\mathcal B_2$ gives
\[
 \norm{Ef}{L^{r_k}(Y)}
 \lessapprox K^{O(1)}M^{1/r_k-1/p}
 \left\|\prod_{i=1}^k|Ef_i|^{1/k}\right\|_{L^p(Y')}.
\]
Since $\gamma(Y')\le\gamma(Y)$ and $1/p-1/r_k\ge0$, the hypothesis \eqref{eq:broad-transfer-hypothesis}, applied to $Y'$, gives the required broad bound. Its logarithmic losses and its losses of powers of $K$ are absorbed into $R^\varepsilon$ by taking first $\eta$ and then $\delta_1$ sufficiently small. Together with the narrow step, this closes the induction. The weighted estimate then follows from Section~\ref{subsec:weighted-reduction}.
\end{proof}

\subsection{Application of the refined Strichartz estimate}\label{subsec:broad-strichartz}
To prove part~\textup{(i)}, let $2\le p\le q$ and apply the selection of Section~\ref{subsec:broad-selection} with $s=q$. This provides functions $f_1,\ldots,f_k$ and a family of cubes $\mathcal B_2$, which we retain.

We assign each cube of $\mathcal B_2$, which we call a \emph{small cube}, to a \emph{parent cube} $Q$ of side comparable to $R^{1/2}$, using a fixed enlargement and a finite coloring to deal with the boundaries. We then pigeonhole the dyadic values of
\[
 \lambda(Q)=\#\{B\in\mathcal B_2:B\text{ is assigned to }Q\},
 \qquad \iota_j(Q)=\norm{Ef_j}{L^q(Q)},\quad 1\le j\le k.
\]
We retain a class of parent cubes that carries a significant proportion of the \emph{small} cubes. Let $\mathcal Q'$ be this class, let $N=\#\mathcal Q'$, let $\lambda$ be the common dyadic value of $\lambda(Q)$ on $\mathcal Q'$, and let $\mathcal B_3$ be the family of small cubes assigned to the cubes of $\mathcal Q'$. Then
\begin{equation}\label{eq:broad-cardinalities}
 \#\mathcal B_3\sim\lambda N
       \gtrapprox K^{-O(1)}M.
\end{equation}
For each $j$, the quantity $\norm{Ef_j}{L^q(Q)}$ is dyadically constant on $\mathcal Q'$, so Theorem~\ref{thm:multilinear} applies to this family.

We apply \eqref{eq:broad-recovery} with $s=q$ and $\mathcal B'=\mathcal B_3$, enlarge the domain of integration to $\bigcup\mathcal Q'$, and use Theorem~\ref{thm:multilinear}. By \eqref{eq:broad-cardinalities}, this yields
\begin{equation}\label{eq:broad-start}
 \norm{Ef}{L^{r_k}(Y)}
 \lessapprox K^{O(1)}R^{\varepsilon/8}
 M^{1/r_k-1/q}(M/\lambda)^{-\frac{k-1}{k(n+1)}}\norm f2.
\end{equation}
The two upper bounds $\lambda\lesssim\min\{M,\gamma R^{m/2}\}$ are both needed. The first is the trivial bound by the total number of cubes, and the second follows from the density condition applied to a parent cube.

Let $S=M/\gamma$. We factor out the quantity $S^{-(\frac1p-\frac1{r_k})}$, which is the factor $(\gamma/M)^{\frac1p-\frac1{r_k}}$ in \eqref{eq:discrete}. The remaining main factor on the right-hand side of \eqref{eq:broad-start} is then bounded by
\[
 \gamma^{1/r_k-1/q}S^{1/p-1/q}
       \min\{1,R^{m/2}/S\}^{\frac{k-1}{k(n+1)}}.
\]
Since $r_k\ge q$ and $\gamma\gtrsim K^{-2m}$, the first factor is at most $K^{O(1)}$. Moreover, $S\lesssim R^{(n+1+m)/2}$ by \eqref{eq:elementary-density}. The product of the other two factors is
\[
 \begin{cases}
 S^{1/p-1/q},&S\le R^{m/2},\\[2pt]
 R^{m(k-1)/(2k(n+1))}S^{1/[k(n+1)]-1/2+1/p},
       &R^{m/2}\le S\lesssim R^{(n+1+m)/2}.
 \end{cases}
\]
Since $0\le\frac12-\frac1p\le\frac1{n+1}$, the maximum of this expression over the admissible values of $S$ is attained at $S=R^{m/2}$ or at $S\sim R^{(n+1+m)/2}$, and the corresponding exponents of $R$ are $C_{n,m}(p)$ and $A_{n,m,k}(p)$. Thus
\[
 \norm{Ef}{L^{r_k}(Y)}
 \lessapprox K^{O(1)}R^{\max\{A_{n,m,k},C_{n,m}\}+\varepsilon/8}
                  (\gamma/M)^{\frac1p-\frac1{r_k}}\norm f2.
\]
We choose $\delta_1$ small enough to absorb the loss $K^{O(1)}$, and then take the base radius large enough to absorb the logarithmic loss. The resulting broad bound, together with the narrow step of Section~\ref{sec:narrow}, closes the induction with $\beta=\max\{A_{n,m,k}(p),N_{n,m,k}(p),C_{n,m}(p)\}$. This proves the case of Proposition~\ref{prop:discrete} that corresponds to part~\textup{(i)} and hence, by Section~\ref{subsec:weighted-reduction}, Theorem~\ref{thm:main}\textup{(i)}.

\subsection{Localized multilinear estimates}\label{subsec:local-interpolation}
Let $n\ge3$, and let $\mathcal Q$ be the family of lattice $R^{1/2}$-cubes meeting $B_R$. For $Q\in\mathcal Q$, let $\nu_Q$ be the measure given by $d\nu_Q=\mathbf1_{Y\cap Q}\,dz$. Then
\begin{equation}\label{eq:parent-frostman}
 [\nu_Q]_m\lesssim K^{2(n+1)}\gamma,\qquad
 \nu_Q(Q)\lesssim K^{2(n+1)}\gamma R^{m/2}.
\end{equation}
To prove the first bound, consider a ball $B(z,r)$. For $K^2\le r\le R^{1/2}$, we count the cubes of $Y$ that meet the ball and multiply by their volume. For $r<K^2$, we use
\[
 \nu_Q(B(z,r))\lesssim r^{n+1}
 \le K^{2(n+1-m)}r^m
 \lesssim K^{2(n+1)}\gamma r^m,
\]
where the last inequality holds because $\gamma\gtrsim K^{-2m}$. For $r>R^{1/2}$, we cover $Q$ by finitely many balls of radius $R^{1/2}$; this proves the second bound in \eqref{eq:parent-frostman} and gives
\[
 \nu_Q(B(z,r))\le\nu_Q(Q)
 \lesssim K^{2(n+1)}\gamma R^{m/2}
 \le K^{2(n+1)}\gamma r^m.
\]
Thus the argument for small radii uses only that $\nu_Q$ has bounded density with respect to Lebesgue measure, and the argument for large radii uses only that $\nu_Q$ is supported in $Q$. Neither requires any further condition on the original weight.

\textbf{Cyclic pairing and interpolation.}
We use the wave packets and the local energies of Section~\ref{subsec:local-packets}, and write
\[
 F_Q=\prod_{i=1}^k|Ef_{i,Q}|^{1/k}.
\]
With the convention $f_{k+1,Q}=f_{1,Q}$, we have the cyclic identity
\[
 F_Q=\prod_{i=1}^k
       \left(|Ef_{i,Q}Ef_{i+1,Q}|^{1/2}\right)^{1/k}.
\]
The $k$-fold transversality implies the pairwise angular separation that is needed in order to apply a bilinear estimate to each of the $k$ factors. By H\"older's inequality, Lemma~\ref{lem:CHL-unit}, and \eqref{eq:parent-frostman},
\begin{equation}\label{eq:local-weighted-endpoint}
 \norm{F_Q}{L^{s_0}(d\nu_Q)}
 \lesssim_\eta K^{O(1)}\gamma^{1/s_0}R^\eta G_Q,
\end{equation}
because $D_{n,m}(s_0)=0$. Similarly, Lemma~\ref{lem:unweighted-multilinear} and the inequality $\nu_Q\le dz$ give
\begin{equation}\label{eq:local-unweighted-endpoint}
 \norm{F_Q}{L^{v_k}(d\nu_Q)}
 \lesssim_\eta K^{O(1)}R^\eta G_Q.
\end{equation}
We also have the pointwise bound $\norm{F_Q}\infty\lesssim G_Q$. Consequently, the bound \eqref{eq:local-unweighted-endpoint} remains valid with $r_k$ in place of $v_k$, with the usual interpretation when $r_k=\infty$.

Suppose first that $s_0\le v_k$. We interpolate between \eqref{eq:local-weighted-endpoint} and \eqref{eq:local-unweighted-endpoint} with the weights $1-v_k/r_k$ and $v_k/r_k$, respectively. By the definition \eqref{eq:local-exponents} of $s_k$, this gives
\begin{equation}\label{eq:local-matched-endpoint}
 \norm{F_Q}{L^{s_k}(d\nu_Q)}
 \lesssim_\eta K^{O(1)}
 \gamma^{1/s_k-1/r_k}R^\eta G_Q.
\end{equation}
If $s_0>v_k$, then $s_k=v_k$, and \eqref{eq:local-matched-endpoint} follows from \eqref{eq:local-unweighted-endpoint} and the bound $\gamma\gtrsim K^{-2m}$.

For $2\le p\le s_k$, we apply H\"older's inequality and the mass bound in \eqref{eq:parent-frostman} to \eqref{eq:local-matched-endpoint}. For $s_k\le p\le r_k$, we interpolate between \eqref{eq:local-matched-endpoint} and the unweighted bound at the exponent $r_k$. In either case,
\begin{equation}\label{eq:local-matched-bound}
 \norm{F_Q}{L^p(Y\cap Q)}
 \lesssim_\eta K^{O(1)}\gamma^{1/p-1/r_k}
 R^{\frac m2(\frac1p-\frac1{s_k})_++\eta}G_Q.
\end{equation}
The exponent of $\gamma$ here is the one required by the discrete induction; compare \eqref{eq:broad-transfer-hypothesis}.

\begin{proposition}[Localized multilinear estimate]\label{prop:halfscale-trilinear}
Let $n\ge3$, $0<m\le n-1$, $3\le k\le n+1$, and $2\le p\le\min\{3,r_k\}$, and let $Y$ be a union of lattice $K^2$-cubes in $B_R$ with density $\gamma$. If $f_1,\ldots,f_k$ are $\nu$-transverse with $\nu\ge K^{-O(1)}$, then, for every $\eta>0$,
\[
 \left\|\prod_{i=1}^k|Ef_i|^{1/k}\right\|_{L^p(Y)} \lesssim_\eta K^{O(1)}\gamma^{1/p-1/r_k}R^{(\frac{k}{2p}-\frac{k-1}{4})_+
      +\frac m2(\frac1p-\frac1{s_k})_++\eta}
 \prod_{i=1}^k\norm{f_i}2^{1/k}.
\]
\end{proposition}
\begin{proof}
On each cube $Q\in\mathcal Q$, we replace $Ef_i$ by $Ef_{i,Q}$ using \eqref{eq:local-energy-approx}, apply \eqref{eq:local-matched-bound}, and then take the $\ell^p$ norm over $Q$. By Lemma~\ref{lem:energy-summation}, this gives the stated power of $R$, provided that we use smaller error exponents in the intermediate estimates and then choose $\delta_2$ sufficiently small.

To handle the error terms in \eqref{eq:local-energy-approx}, we may assume that the norms $\norm{f_i}2$ are nonzero and, by homogeneity, that each of them equals one. The other factors and the number of cubes are bounded by powers of $R$, whereas the error terms decay to arbitrary order. Since $\gamma\gtrsim K^{-2m}$, the error terms can therefore be absorbed into the right-hand side.
\end{proof}
Proposition~\ref{prop:halfscale-trilinear} verifies the hypothesis \eqref{eq:broad-transfer-hypothesis} of Proposition~\ref{prop:broad-transfer}, which therefore yields Theorem~\ref{thm:main}\textup{(iii)}. Together with Section~\ref{sec:discrete}, this completes the proof of Theorem~\ref{thm:main}.

\section{Optimization}\label{sec:optimization}
For fixed $n$, $m$, and $k$, each exponent in Theorem~\ref{thm:main} is a continuous, piecewise affine function of $1/p$ on its admissible interval, and its breakpoints are intersection points of finitely many affine functions. Hence $B_{n,m}$ in \eqref{eq:convex} can be computed by listing these intersection points and the endpoints of the intervals, together with the point $(1/p,\beta)=(0,0)$, and taking the lower convex hull. For each fixed $n$ and $m$, this is a finite and exact computation.

\subsection{The small-\texorpdfstring{$p$}{p} estimate}
Take $k=n+1$ in Theorem~\ref{thm:full}\textup{(i)}. We have
\[
\begin{aligned}
 A_{n,m,n+1}-N_{n,m,n+1}
   &=\frac{n-m}{2n(n+1)}-\frac n2\left(\frac12-\frac1p\right),\\
 A_{n,m,n+1}-C_{n,m}
   &=\frac1{2(n+1)}-\frac{n+1}{2}\left(\frac12-\frac1p\right).
\end{aligned}
\]
Both differences are nonnegative for $2\le p\le p_{n,m}$, that is, when
\[
 0\le\frac12-\frac1p
 \le\min\left\{\frac{n-m}{n^2(n+1)},\frac1{(n+1)^2}\right\}.
\]
In this range the maximum in part~\textup{(i)} is therefore
\[
 A_{n,m,n+1}(p)
 =\frac{m+1}{2(n+1)}-\frac{n+1+m}{2}\left(\frac12-\frac1p\right),
\]
which proves Corollary~\ref{cor:small-p-general}. When $m\ge1$, the first term in the minimum in \eqref{eq:small-p-general-endpoint} is the smaller one, which gives the simplified formula for $p_{n,m}$.

\subsection{Zero-power ranges and the choice of \texorpdfstring{$k$}{k}}\label{subsec:zero-optimization}
Throughout this subsection we assume that $m>0$. The exponent in Theorem~\ref{thm:main}\textup{(ii)} vanishes precisely when
\[
 p\ge\max\left\{\bar p_k,
 \frac{2(m+1)}{m-1+2(n-m)/(k-1)}\right\}.
\]
The denominator in the second term is positive. This threshold is at most $r_k$, because $\bar p_k<r_k$ and
\[
 \frac{m-1+2(n-m)/(k-1)}{2(m+1)}-\frac1{r_k}
 =\frac{n+2-k}{(m+1)(k-1)}>0.
\]
Interpolating with the $L^\infty$ bound \eqref{eq:Linfty} then proves the estimate of Corollary~\ref{cor:zero-power} for all larger $p$.

The first term in the maximum decreases as $k$ increases, whereas the second increases. Writing $d=k-1$ and treating $d$ as a real variable, we find that the two terms are equal when
\[
 \frac{2(n+d+2)}{n+d}
 =\frac{2(m+1)}{m-1+2(n-m)/d},
 \qquad d^2+d=(n-m)(n+2).
\]
Thus, to compute $p_\circ$, it suffices to check the integers adjacent to
\[
 d_* =\frac{\sqrt{1+4(n-m)(n+2)}-1}{2},
\]
restricted to the range $2\le d\le n$. For $n=2$ the only choice is $k=3$, and $p_\circ=\max\{3,2(m+1)\}$. For $n\ge4$, taking $k=3$ gives
\[
 N_{n,m,3}(\bar p_3)
 =\frac{(m+1)(n+2)}{4(n+4)}-\frac{n-1}{4}
 \le\frac{4-n}{4(n+4)}\le0,
\]
so the exponent in part~\textup{(ii)} vanishes at $p=\bar p_3=2(n+4)/(n+2)$. This proves the explicit range in Corollary~\ref{cor:zero-power}.

The exponent in part~\textup{(i)} is positive for $p<q$, because $C_{n,m}(p)>0$ there. The exponent in part~\textup{(iii)} can vanish only if $p\ge2k/(k-1)$ and $N_{n,m,k}(p)\le0$; since $\bar p_k\le2k/(k-1)$, this range is already covered by part~\textup{(ii)}. Finally, all the exponents are nonnegative, so interpolation cannot move the left endpoint of the interval on which the exponent is zero. Hence $p_\circ$ is the smallest value of $p$ at which the envelope $B_{n,m}$ of \eqref{eq:convex} vanishes.

Some values of $p_\circ$ are listed in the following table.
\begin{center}
\begin{tabular}{cccc}
\toprule
Ambient space & $m$ & A choice of $k$ & $p_\circ$\\
\midrule
$\R^4$ & $1/2$ & $4$ & $8/3$\\
$\R^4$ & $2$ & $3$ & $3$\\
$\R^5$ & $2$ & $4$ & $18/7$\\
$\R^5$ & $3$ & $3$ & $8/3$\\
$\R^6$ & $2$ & $5$ & $22/9$\\
$\R^6$ & $4$ & $3$ & $18/7$\\
$\R^7$ & $2$ & $6$ & $26/11$\\
$\R^9$ & $1$ & $8$ & $34/15$\\
\bottomrule
\end{tabular}
\end{center}
These are the thresholds given by the estimates proved here. We show that the threshold cannot be lowered only in the cases specified in Corollary~\ref{cor:zero-power} and Section~\ref{subsec:sharpness}.

\subsection{Some lower-dimensional results}\label{subsec:refined-optimization}
For $n=2$ and $m=1$, the relevant vertices in the $(1/p,\beta)$ plane are
\[
 (1/2,1/3),\quad(5/12,1/6),\quad(1/4,0),\quad(0,0).
\]
The first two come from Theorem~\ref{thm:main}\textup{(i)} with $k=3$, and the zero exponent at $p=4$ comes from part~\textup{(ii)}. The lower convex envelope is
\[
 B_{2,1}(p)=
 \begin{cases}
  2/p-2/3,&2\le p\le12/5,\\[2pt]
  1/p-1/4,&12/5\le p\le4,\\[2pt]
  0,&p\ge4.
 \end{cases}
\]

For $n=3$, $m=2$, and $k=3$, the broad and narrow exponents in part~\textup{(ii)} are
\[
 \left(\frac3p-\frac{15}{14}\right)_+,
 \qquad N_{3,2,3}(p)=\frac3{2p}-\frac12.
\]
They are equal at $p=21/8$, where their common value is $1/14$. Together with the vertices for small $p$ that come from part~\textup{(i)}, the vertices of the envelope are
\[
 (1/2,3/8),\quad(17/36,7/24),\quad
 (8/21,1/14),\quad(1/3,0),\quad(0,0).
\]
Consequently,
\begin{equation}\label{eq:B3-refined}
 B_{3,2}(p)=
 \begin{cases}
  3/p-9/8,&2\le p\le36/17,\\[2pt]
  111/(46p)-39/46,&36/17\le p\le21/8,\\[2pt]
  3/(2p)-1/2,&21/8\le p\le3,\\[2pt]
  0,&p\ge3.
 \end{cases}
\end{equation}
Here the line through $(17/36,7/24)$ and $(8/21,1/14)$ has slope $111/46$ and intercept $-39/46$. In particular, $B_{3,2}(12/5)=29/184$, which proves \eqref{eq:R4-refined-originalp}. The formula \eqref{eq:B3-refined} also gives the upper bound in Corollary~\ref{cor:sharp-R4}.

For $n=5$ and $m=4$, the vertices with positive exponent are
\[
 (1/2,5/12),\quad(37/75,23/60),\quad
 (143/300,19/60),\quad(11/24,1/4).
\]
They come from part~\textup{(i)} with $k=6,6,5,4$, respectively, and part~\textup{(ii)} with $k=3$ supplies the vertex $(7/18,0)$. Taking the lower convex hull gives
\[
 B_{5,4}(p)=
 \begin{cases}
  5/p-25/12,&2\le p\le75/37,\\[2pt]
  4/p-159/100,&75/37\le p\le300/143,\\[2pt]
  40/(11p)-17/12,&300/143\le p\le24/11,\\[2pt]
  18/(5p)-7/5,&24/11\le p\le18/7,\\[2pt]
  0,&p\ge18/7.
 \end{cases}
\]
In particular, $B_{5,4}(12/5)=1/10$ and $B_{5,4}(18/7)=0$.

The localized estimate of part~\textup{(iii)} contributes in fractional weight dimensions. For example, when $n=4$, $m=3/4$, and $k=5$, we have
\[
 r_5=4,\qquad v_5=5/2,\qquad s_0=2,\qquad s_5=16/7.
\]
At $p=52/21$, the terms in the exponent of part~\textup{(iii)} are
\[
 \frac5{2p}-1=\frac1{104}=N_{4,3/4,5}(p),
 \qquad \left(\frac1p-\frac1{s_5}\right)_+=0.
\]
Thus part~\textup{(iii)} supplies the vertex $(21/52,1/104)$ of $B_{4,3/4}$.

The following table records some values of the envelope. In every entry, the estimate holds up to an $R^\varepsilon$ loss.
\begin{center}
\begin{tabular}{cccc}
\toprule
Ambient space & $m$ & $p$ & $B_{n,m}(p)$\\
\midrule
$\R^3$ & $1$ & $7/2$ & $1/28$\\
$\R^4$ & $2$ & $12/5$ & $29/184$\\
$\R^4$ & $2$ & $21/8$ & $1/14$\\
$\R^5$ & $3$ & $8/3$ & $0$\\
$\R^5$ & $3/4$ & $52/21$ & $1/104$\\
$\R^6$ & $4$ & $12/5$ & $1/10$\\
$\R^6$ & $4$ & $18/7$ & $0$\\
$\R^6$ & $3$ & $7/3$ & $3/35$\\
\bottomrule
\end{tabular}
\end{center}
For the last entry, the relevant segment of the envelope joins $(7/15,1/5)$, supplied by part~\textup{(i)} with $k=5$, to $(2/5,0)$, supplied by part~\textup{(ii)} with $k=4$; on this segment the exponent is $3/p-6/5$. The sharpness statements are proved in Section~\ref{subsec:sharpness}.

\section{Necessary conditions}\label{sec:examples}\label{sec:counterexample}
We now construct examples which show that certain losses are unavoidable under the half-scale hypothesis \eqref{eq:weight}. In each construction, $R$ tends to infinity along a sequence, $H_R$ is the indicator function of a subset of $B_R$, and the constant $C_{H_R}$ in \eqref{eq:weight} is bounded independently of $R$. Consequently, if
\begin{equation}\label{eq:necessary-assumption}
 \norm{Ef}{L^p(B_R,H)}
 \lesssim_\varepsilon C_H^{1/p}R^{\beta+\varepsilon}\norm f2
\end{equation}
holds for every $\varepsilon>0$, then a lower bound $\norm{Ef_R}{L^p(H_R)}/\norm{f_R}2\gtrsim R^b$ forces $\beta\ge b$. Throughout this section we restrict to $p\ge2$.

For convenience, set
\[
\begin{aligned}
 P_{n,m}(p)&=\frac{m+1}{2p}-\frac{n-1}{4},\\
 L_{n,m}(p)&=\frac{m}{2(n+1)}
       -\frac{n+1+m}{2}\left(\frac12-\frac1p\right).
\end{aligned}
\]
For the hybrid examples, we also write
\begin{equation}\label{eq:hybrid-obstruction}
 \mathfrak b_{n,m,j}(p)=\frac{m-j}{2(n+1-j)}
       -\frac{n+1+m-j}{2}\left(\frac12-\frac1p\right),
\end{equation}
where $j$ ranges over
\[
 \mathcal J_{n,m}=\{j\in\mathbb Z:0\le j\le m,\ n-j\ge4\}.
\]

\begin{proposition}[Necessary conditions]\label{prop:lower-bounds}
Let $n\ge2$ and $0<m\le n-1$. Every exponent $\beta$ for which \eqref{eq:necessary-assumption} holds satisfies
\begin{equation}\label{eq:universal-lower}
 \beta\ge\max\{0,P_{n,m}(p)\}.
\end{equation}
If $n\ge3$ and $(n+1)/2\le m\le n-1$, it also satisfies $\beta\ge L_{n,m}(p)$. Moreover, for every $j\in\mathcal J_{n,m}$,
\begin{equation}\label{eq:hybrid-lower}
 \beta\ge\mathfrak b_{n,m,j}(p).
\end{equation}
In particular, the case $j=0$ extends the lattice bound $\beta\ge L_{n,m}(p)$ to every $0<m\le n-1$ when $n\ge4$.
\end{proposition}
When $\mathcal J_{n,m}$ is empty, the last assertion is vacuous. The proposition is proved by the constructions in the following subsections. The bound $\beta\ge0$ is elementary: it follows by taking a fixed nonnegative function $f$ and letting $H$ be the indicator function of a sufficiently small fixed ball. We refer to this as the \emph{fixed-ball example}.

\subsection{A lightplank for every weight dimension}\label{subsec:lightplank}
We use the null coordinates of Section~\ref{subsec:charts}. Fix a sufficiently small constant $c>0$, depending only on the dimension, and let
\[
 \Omega_R=\{(\eta,\rho):|\eta|\le cR^{-1/2},\ 3\le\rho\le3+c\},
 \qquad f_R=\mathbf1_{\Omega_R}.
\]
These frequencies lie in a fixed chart of the original annulus $1\le|\xi|\le2$. The corresponding function in the original coordinates is obtained by undoing the change of variables together with its smooth Jacobian factor, and its $L^2$ norm is comparable to $\norm{f_R}2$. Similarly, pulling a weight back under the fixed linear change of physical coordinates preserves the ball condition up to fixed constants. If the supports below are chosen sufficiently small, the support of the weight in the original physical coordinates is contained in $B_R$.

On the lightplank
\[
 T_R=\{(s,y,v):|s|\le cR,\ |y|\le c\sqrt R,\ |v|\le c\},
\]
the phase $s|\eta|^2/(2\rho)+y\cdot\eta+v\rho$ has absolute value $O(c)$ for all $(\eta,\rho)\in\Omega_R$. Taking the real part of the integral that defines $Ef_R$, we obtain
\begin{equation}\label{eq:plank-phase}
 |Ef_R|\gtrsim|\Omega_R|\sim R^{-(n-1)/2}
       \quad\text{on }T_R,
 \qquad \norm{f_R}2\sim R^{-(n-1)/4}.
\end{equation}

Suppose first that $1\le m\le n-1$. Put
\[
 D_R=R^{(n-m)/(2(n-1))},\qquad
 A_R=(D_R\mathbb Z^{n-1}+[-c,c]^{n-1})
            \cap[-c\sqrt R,c\sqrt R]^{n-1}.
\]
After decreasing $c$ in the definition of $A_R$ if necessary, define
\[
 H_R(s,y,v)=\mathbf1_{\{|s|\le cR\}}
             \mathbf1_{A_R}(y)\mathbf1_{\{|v|\le c\}},
 \qquad \supp H_R\subseteq T_R.
\]
Then
\[
 \int H_R\sim R\left(\frac{\sqrt R}{D_R}\right)^{n-1}
          =R^{(m+1)/2}.
\]
For $1\le r\le\sqrt R$, a ball of radius $r$ meets at most $O((1+r/D_R)^{n-1})$ of the transverse boxes that make up $A_R$, and therefore
\[
 \int_{B(z,r)}H_R\lesssim r(1+r/D_R)^{n-1}\lesssim r^m.
\]
Indeed, for $r\le D_R$ we use $r\le r^m$, and for $r\ge D_R$ we use
\[
 r^nD_R^{-(n-1)}=r^nR^{-(n-m)/2}\le r^m.
\]

If $0<m<1$, we put instead $D_R=R^{(1-m)/2}$ and let $H_R$ be the indicator function of a union of boxes of unit size, centered at the points of $D_R\mathbb Z$ on the $s$-axis and contained in the region where $|s|\le cR$, $|y|\le c$, and $|v|\le c$. Then $H_R$ has mass $\sim R/D_R=R^{(m+1)/2}$, and
\[
 \int_{B(z,r)}H_R\lesssim1+r/D_R\lesssim r^m
       \qquad(1\le r\le\sqrt R).
\]
In both cases, \eqref{eq:plank-phase} gives
\begin{equation}\label{eq:plank-ratio}
 \frac{\norm{Ef_R}{L^p(H_R)}}{\norm{f_R}2}
 \gtrsim R^{(m+1)/(2p)-(n-1)/4}.
\end{equation}
Together with the fixed-ball example, this proves \eqref{eq:universal-lower}.

\subsection{A lattice in null coordinates}\label{subsec:null-lattice}
This example is based on the phase-alignment principle used for the paraboloid in \cite[Section 3]{DLWZ}. Assume that
\[
 n\ge3,\qquad (n+1)/2\le m\le n-1,
\]
and choose
\[
 \kappa=\frac{n+1-m}{2(n+1)},\qquad Q=R^\kappa\in\mathbb N.
\]
Here $R$ runs through a sequence of radii for which $Q$ is an integer. Since $m\ge(n+1)/2$, we have $Q^2\le\sqrt R$. Let $c,c_0>0$ be sufficiently small fixed constants and, in the chart \eqref{eq:null-operator}, define
\[
 \begin{aligned}
 \Omega_R={}&\left[(Q^{-1}\mathbb Z^{n-1}+B^{n-1}(0,cR^{-1}))
                                  \cap B^{n-1}(0,c_0)\right]\\
             &\hspace{25mm}\times[3-cR^{-1},3+cR^{-1}],
 \qquad f_R=\mathbf1_{\Omega_R}.
 \end{aligned}
\]
Then
\begin{equation}\label{eq:lattice-frequency-volume}
 |\Omega_R|\sim Q^{n-1}R^{-n}.
\end{equation}
Let $\Lambda_R$ be the set of points $(s,y,v)\in[-cR,cR]^{n+1}$ that satisfy
\begin{equation}\label{eq:lattice-physical}
 \dist(s,12\pi Q^2\mathbb Z)\le c,\qquad
 \dist(y,2\pi Q\mathbb Z^{n-1})\le c,
\end{equation}
with no further restriction on $v$, and put $H_R=\mathbf1_{\Lambda_R}$. Changing the shape of the unit neighborhoods by fixed factors does not affect any of the counts below. Then
\[
 |\Lambda_R|\sim R^{n+1}Q^{-(n+1)}.
\]

At a central frequency $(\eta,\rho)=(h/Q,3)$ with $h\in\mathbb Z^{n-1}$, and at central coordinates $s=12\pi Q^2b$ and $y=2\pi Qa$ with $b\in\mathbb Z$ and $a\in\mathbb Z^{n-1}$, the phase equals
\[
 2\pi b|h|^2+2\pi a\cdot h+3v\in2\pi\mathbb Z+3v.
\]
The frequency perturbations allowed in $\Omega_R$ are $O(R^{-1})$, the physical coordinates are $O(R)$, and the physical perturbations in \eqref{eq:lattice-physical} are $O(c)$. Consequently, after multiplication by the common factor $e^{-3iv}$, the integrand has real part bounded below by a positive constant throughout $\Omega_R$. Thus
\begin{equation}\label{eq:lattice-phase}
 |Ef_R|\gtrsim|\Omega_R|\quad\text{on }\Lambda_R.
\end{equation}

We now verify the ball condition. For $1\le r\le\sqrt R$,
\[
 |\Lambda_R\cap B(z,r)|
 \lesssim r(1+r/Q)^{n-1}(1+r/Q^2).
\]
For $r\le Q$ the right-hand side is $O(r)$, and hence $O(r^m)$. In the range $Q\le r\le Q^2$,
\[
 \frac{|\Lambda_R\cap B(z,r)|}{r^m}
 \lesssim r^{n-m}Q^{-(n-1)}
 \le Q^{n+1-2m}\le1.
\]
For $Q^2\le r\le\sqrt R$,
\[
 \frac{|\Lambda_R\cap B(z,r)|}{r^m}
 \lesssim r^{n+1-m}Q^{-(n+1)}
 \le R^{(n+1-m)/2}Q^{-(n+1)}=1.
\]
Hence $C_{H_R}=O(1)$ uniformly in $R$. Using \eqref{eq:lattice-frequency-volume}, \eqref{eq:lattice-phase}, and the value of $|\Lambda_R|$, we obtain
\[
 \frac{\norm{Ef_R}{L^p(H_R)}}{\norm{f_R}2}
 \gtrsim (Q^{n-1}R^{-n})^{1/2}
                 (R^{n+1}Q^{-(n+1)})^{1/p}=R^{L_{n,m}(p)}.
\]
This proves the bound $\beta\ge L_{n,m}(p)$ of Proposition~\ref{prop:lower-bounds}.

\subsection{The hybrid lattice--Knapp example}\label{subsec:hybrid}\label{subsec:integer-cone}
We begin with an elementary count of integer points on a cone, which will be used in the construction. Related lattice constructions on cones appear in \cite[Proposition 5.5]{HarrisRS} and \cite[Section 5]{DHL}. The restriction to $d\ge4$ allows a direct argument based on Jacobi's four-square theorem.

\begin{lemma}\label{lem:integer-cone}
For $d\ge4$ and every sufficiently large integer $Q$, there is a set
\[
 \mathcal P_Q\subseteq
 \{(h,b)\in\mathbb Z^d\times\mathbb Z:
      |h|^2=b^2,\ 5Q/4\le b\le3Q/2\}
\]
with $\#\mathcal P_Q\sim Q^{d-1}$.
\end{lemma}
\begin{proof}
Let $r_4(N)$ denote the number of representations of a positive integer $N$ as a sum of four squares of integers, where representations that differ in the order or in the signs of the terms are counted as distinct. For odd $N$, Jacobi's four-square theorem \cite[Section 27.13(iv)]{DLMF} gives
\[
 r_4(N)=8\sum_{a\mid N}a\ge8N,
\]
where the sum is over the positive divisors of $N$.
Choose an odd integer $b\in[5Q/4,3Q/2]$ and a vector
\[
 z\in(2\mathbb Z)^{d-4}\cap B(0,Q/10).
\]
Then $N=b^2-|z|^2$ is odd and comparable to $Q^2$. Hence there are $\gtrsim Q^2$ vectors $a\in\mathbb Z^4$ with $|a|^2=N$, and for each of them $h=(a,z)$ satisfies $|h|^2=b^2$. There are $\sim Q$ choices of $b$ and $\sim Q^{d-4}$ choices of $z$ (when $d=4$, $z$ is the empty vector), so we obtain $\gtrsim Q^{d-1}$ distinct pairs $(h,b)$. We let $\mathcal P_Q$ be a set of $\sim Q^{d-1}$ of these pairs.
\end{proof}

We now construct the hybrid example. Fix $j\in\mathcal J_{n,m}$ and put $d=n-j\ge4$. Write
\[
 \xi=(\zeta,\eta)\in\R^j\times\R^d,
 \qquad x=(x',x'')\in\R^j\times\R^d.
\]
The lattice is placed in the $d$ variables $\eta$, while the remaining $j$ frequency variables $\zeta$ are confined to scale $R^{-1/2}$, as in a Knapp example. This is the mechanism of a lattice on a lower-dimensional cone from \cite[Section 5]{DHL}; the spacing of the lattice chosen below is dictated by our half-scale ball condition.

Choose
\[
 \kappa=\frac{n+1-m}{2(n+1-j)}
       =\frac{n+1-m}{2(d+1)},\qquad Q=R^\kappa\in\mathbb N.
\]
As before, $R$ runs through a sequence of radii for which $Q$ is an integer. Since $j\le m$, we have $0<\kappa\le1/2$. Let $\mathcal P_Q$ be the set provided by Lemma~\ref{lem:integer-cone}. For a sufficiently small fixed $c>0$, define
\[
 \begin{gathered}
 \Omega_R'=\bigcup_{(h,b)\in\mathcal P_Q}
                       B^d(h/Q,cR^{-1}),\\
 \Omega_R=[-cR^{-1/2},cR^{-1/2}]^j\times\Omega_R',
 \qquad f_R=\mathbf1_{\Omega_R}.
 \end{gathered}
\]
The balls that make up $\Omega_R'$ are pairwise disjoint, because distinct centers are at distance at least $Q^{-1}$ from each other and $R^{-1}\ll Q^{-1}$. Moreover, $|h/Q|=b/Q\in[5/4,3/2]$, so $\Omega_R\subset\{1\le|\xi|\le2\}$ when $R$ is large. Thus
\begin{equation}\label{eq:hybrid-frequency-volume}
 |\Omega_R|\sim R^{-j/2}Q^{d-1}R^{-d},
 \qquad \norm{f_R}2=|\Omega_R|^{1/2}.
\end{equation}

In the physical variables $(x'',t)$, set
\[
 \begin{gathered}
 \Lambda_R'=\bigcup_{\substack{(v,w)\in\mathbb Z^d\times\mathbb Z\\
                       |2\pi Q \cdot (v,w)|_\infty\le cR}}
                  \bigl(2\pi Q \cdot (v,w)+[-c,c]^{d+1}\bigr),\\
 \Lambda_R=[-c\sqrt R,c\sqrt R]^j\times\Lambda_R',
 \qquad H_R=\mathbf1_{\Lambda_R}.
 \end{gathered}
\]
We choose $c$ so small that $\Lambda_R\subset B_R$. The measure of $\Lambda_R$ is
\begin{equation}\label{eq:hybrid-physical-volume}
 |\Lambda_R|\sim R^{j/2}(R/Q)^{d+1}
                  =R^{(n+1+m-j)/2}.
\end{equation}
Here and below, when $j=0$ the zero-dimensional boxes are understood to have measure one.

\paragraph{The half-scale ball condition.}
For $1\le r\le\sqrt R$,
\[
 |\Lambda_R\cap B(z,r)|
       \lesssim r^j(1+r/Q)^{d+1}\lesssim r^m.
\]
Indeed, for $r\le Q$ this follows from $j\le m$, and for $Q\le r\le\sqrt R$ it follows from
\[
 \frac{r^{j+d+1}}{Q^{d+1}}
 =\frac{r^{n+1}}{R^{(n+1-m)/2}}\le r^m.
\]
Thus $C_{H_R}=O(1)$ uniformly in $R$, also when $m$ is not an integer.

\paragraph{Constructive interference.}
At the central frequencies $\xi=(0,h/Q)$ and at the lattice points $(x'',t)=2\pi Q(v,w)$, the phase equals
\[
 (2\pi Qv)\cdot(h/Q)+(2\pi Qw)|h/Q|
          =2\pi(v\cdot h+wb)\in2\pi\mathbb Z.
\]
For $\eta=h/Q+e$ with $|e|\le cR^{-1}$ and $|\zeta|\lesssim cR^{-1/2}$, we have
\[
 \left|\sqrt{|\eta|^2+|\zeta|^2}-|h/Q|\right|
 \le|e|+
       \frac{|\zeta|^2}{\sqrt{|\eta|^2+|\zeta|^2}+|\eta|}\lesssim cR^{-1}+c^2R^{-1}.
\]
After multiplication by $|t|\lesssim cR$, this error remains small. The errors caused by the $O(c)$ perturbations of the physical variables and by the term $x'\cdot\zeta$ are also $O(c)$. Hence, uniformly on $\Lambda_R\times\Omega_R$,
\[
 x\cdot\xi+t|\xi|\in2\pi\mathbb Z+[-O(c),O(c)].
\]
Taking $c$ sufficiently small and taking the real part of the integral, we obtain
\begin{equation}\label{eq:hybrid-phase}
 |Ef_R|\gtrsim|\Omega_R|\quad\text{on }\Lambda_R.
\end{equation}

\paragraph{The exponent.}
By \eqref{eq:hybrid-frequency-volume}, \eqref{eq:hybrid-physical-volume}, and \eqref{eq:hybrid-phase},
\[
 \begin{aligned}
 \frac{\norm{Ef_R}{L^p(H_R)}}{\norm{f_R}2}
 &\gtrsim
   \bigl(R^{-j/2}Q^{d-1}R^{-d}\bigr)^{1/2}
   \bigl(R^{j/2}(R/Q)^{d+1}\bigr)^{1/p}\\
 &=R^{-j/4-d/2+\kappa(d-1)/2+(n+1+m-j)/(2p)}=R^{\mathfrak b_{n,m,j}(p)}.
 \end{aligned}
\]
This proves \eqref{eq:hybrid-lower} and completes the proof of Proposition~\ref{prop:lower-bounds}. When $j=0$, the construction is an isotropic lattice in physical space, which is available for all $0<m\le n-1$ when $n\ge4$; unlike the construction in null coordinates, it does not require $m\ge(n+1)/2$.

\subsection{Comparison and optimization of the examples}\label{subsec:hybrid-comparison}
When $j=0$, the exponent \eqref{eq:hybrid-obstruction} reduces to $L_{n,m}$. For every $j\in\mathcal J_{n,m}$,
\[
 \mathfrak b_{n,m,j}(p)-L_{n,m}(p)
 =\frac j2\left(\frac12-\frac1p
       -\frac{n+1-m}{(n+1)(n+1-j)}\right).
\]
Thus a positive $j$ improves on the full-dimensional lattice whenever the expression in parentheses is positive. To optimize over $j$, write $D=n+1-j$. Then
\begin{equation}\label{eq:hybrid-optimization}
 \mathfrak b_{n,m,j}(p)
 =\frac12-\frac m2\left(\frac12-\frac1p\right)
       -\left[\frac{n+1-m}{2D}+\frac D2\left(\frac12-\frac1p\right)\right].
\end{equation}
For $p>2$, the bracketed expression, regarded as a function of a real variable $D>0$, is minimized at
\[
 D_* =\sqrt{\frac{n+1-m}{1/2-1/p}}.
\]
The maximum of $\mathfrak b_{n,m,j}(p)$ over $j\in\mathcal J_{n,m}$ is therefore found by checking the admissible integers nearest to $D_*$, or the nearest endpoint if $D_*$ lies outside the admissible range
\[
 \max\{5,n+1-\lfloor m\rfloor\}\le D\le n+1.
\]
This optimization is carried out only within the family of examples constructed above.

The hybrid bound can exceed both $L_{n,m}$ and $P_{n,m}$, as the following exact values show.
\begin{center}
\begin{tabular}{cccccc}
\toprule
Ambient space & $m$ & $p$ & $j$ & $\max\{0,L_{n,m},P_{n,m}\}$ & $\mathfrak b_{n,m,j}$\\
\midrule
$\R^9$ & $7$ & $144/67$ & $1$ & $1/9$ & $11/96$\\
$\R^{10}$ & $8$ & $15/7$ & $2$ & $1/10$ & $13/120$\\
$\R^{21}$ & $19$ & $840/403$ & $11$ & $1/21$ & $179/1680$\\
\bottomrule
\end{tabular}
\end{center}
For example, when $(n,m,j)=(9,8,2)$ we have $Q=R^{1/8}$, $|\Omega_R|\sim R^{-29/4}$, and $|\Lambda_R|\sim R^8$, so that at $p=15/7$ the lower bound $|\Omega_R|^{1/2}|\Lambda_R|^{1/p}$ is $R^{-29/8+56/15}=R^{13/120}$.

More generally, for $n\ge8$ let
\[
 m=n-1,\qquad p_* =\frac{2n(n+1)}{n^2+3}.
\]
Then $L_{n,n-1}(p_*)=P_{n,n-1}(p_*)=1/(n+1)$, whereas for $j=1$ we have
\[
 \mathfrak b_{n,n-1,1}(p_*)
 =\frac1{n+1}+\frac{n-7}{4n(n+1)}>\frac1{n+1}.
\]
Optimizing over $j$ gives $D_*\sim2\sqrt n$ and, as $n\to\infty$,
\[
 \max_{j\in\mathcal J_{n,n-1}}\mathfrak b_{n,n-1,j}(p_*)
       =\frac14-\frac1{\sqrt n}+O(1/n).
\]
This follows by substituting $D_*$ into \eqref{eq:hybrid-optimization}; rounding $D_*$ to the nearest admissible integer changes the value by $O(n^{-3/2})$.

At $p=2$, however,
\[
 \mathfrak b_{n,m,j}(2)-L_{n,m}(2)
 =-\frac{j(n+1-m)}{2(n+1)(n+1-j)}\le0.
\]
Thus, at $p=2$, the hybrid family gives no improvement over the full-dimensional lattice.

\subsection{Sharp ranges and the remaining gaps}\label{subsec:sharpness}
We first prove Corollary~\ref{cor:sharp-R3}. When $n=2$, the only choice is $k=3$, and
\[
\begin{aligned}
 N_{2,m,3}(p)&=\frac{m+1}{2p}-\frac14=P_{2,m}(p),\\
 A_{2,m,3}(p)-N_{2,m,3}(p)&=\frac1p-\frac{m+4}{12},\\
 N_{2,m,3}(p)-C_{2,m}(p)&=\frac1{2p}+\frac{m-3}{12}.
\end{aligned}
\]
Suppose first that $2/3\le m\le1$. Then part~\textup{(i)} of Theorem~\ref{thm:main} gives the exponent $P_{2,m}(p)$ at $p=12/(m+4)$. Since $\bar p_3=3$ and $m\le1$, the exponent in part~\textup{(ii)} is
\[
 \max\left\{P_{2,m}(p),\,
       \frac{m+3}{2}\left(\frac1p-\frac13\right)_+,\,0\right\}.
\]
It vanishes at $p=2(m+1)\ge3$. Interpolating between these two values of $p$ gives the exponent $P_{2,m}(p)$ throughout the range
\[
 12/(m+4)\le p\le2(m+1).
\]
Now suppose that $1/2\le m<2/3$. For $p\le3$, the difference between the broad exponent in part~\textup{(ii)} and $P_{2,m}$ is
\[
 \frac{m+3}{2}\left(\frac1p-\frac13\right)
       -P_{2,m}(p)=\frac1p-\frac{2m+3}{12}.
\]
For $p\ge3$, the broad exponent is zero. Thus part~\textup{(ii)} gives the exponent $P_{2,m}(p)$ for
\[
 12/(2m+3)\le p\le2(m+1).
\]
In both cases the interval is nonempty; when $m=1/2$ it reduces to the single point $p=3$. The lightplank bound \eqref{eq:plank-ratio} shows that the positive exponent $P_{2,m}(p)$ is sharp, and the fixed-ball example shows that the zero exponent is sharp for larger $p$. This proves Corollary~\ref{cor:sharp-R3}.

For Corollary~\ref{cor:sharp-R4}, we use \eqref{eq:B3-refined}. For $21/8\le p\le3$, the exponent $3/(2p)-1/2$ coincides with the exponent in the lightplank bound \eqref{eq:plank-ratio}, and for $p\ge3$ the zero exponent is matched by the fixed-ball example. When $n=4$ and $m=3$, Corollary~\ref{cor:zero-power} gives the zero exponent for $p\ge8/3$, whereas the lightplank bound gives the positive exponent $2/p-3/4$ for $p<8/3$; hence the threshold $8/3$ cannot be lowered. In general, the fixed-ball example shows that the zero exponent is sharp wherever it holds, but it does not determine the smallest $p$ for which it holds.

The exponent in the small-$p$ estimate of Corollary~\ref{cor:small-p-general} exceeds the lattice exponent by a constant:
\[
 \left[\frac{m+1}{2(n+1)}
       -\frac{n+1+m}{2}\left(\frac12-\frac1p\right)\right]-L_{n,m}(p)
       =\frac1{2(n+1)}.
\]
This comparison involves only the bound of Corollary~\ref{cor:small-p-general} and the lattice example; it takes into account neither the other upper bounds nor the hybrid examples. The following table compares some of our upper exponents with the lower bounds proved in this section.
\begin{center}
\begin{tabular}{ccccc}
\toprule
Ambient space & $m$ & $p$ & Upper exponent & Lower exponent\\
\midrule
$\R^3$ & $1$ & $12/5$ & $1/6$ & $1/6$\\
$\R^3$ & $1$ & $7/2$ & $1/28$ & $1/28$\\
$\R^3$ & $1$ & $4$ & $0$ & $0$\\
$\R^4$ & $2$ & $2$ & $3/8$ & $1/4$\\
$\R^4$ & $2$ & $12/5$ & $29/184$ & $1/8$\\
$\R^4$ & $2$ & $21/8$ & $1/14$ & $1/14$\\
$\R^5$ & $3$ & $8/3$ & $0$ & $0$\\
$\R^6$ & $4$ & $2$ & $5/12$ & $1/3$\\
$\R^6$ & $4$ & $12/5$ & $1/10$ & $1/24$\\
$\R^6$ & $4$ & $18/7$ & $0$ & $0$\\
$\R^6$ & $3$ & $7/3$ & $3/35$ & $0$\\
\bottomrule
\end{tabular}
\end{center}
For instance, in $\R^4$ with $m=2$ and $p=12/5$, the two exponents differ by $29/184-1/8=3/92$, so we do not know whether the exponent $29/184$ is sharp.

\end{document}